\documentclass[a4paper, 10pt]{amsart}

\usepackage[T1]{fontenc}
\usepackage[utf8]{inputenc}

\usepackage{amssymb}
\usepackage{amsmath}
\usepackage{amsfonts}
\usepackage{amsthm}
\usepackage{mathtools}

\usepackage{fixltx2e} 
\usepackage{isomath} 

\usepackage{graphicx, xcolor}

\usepackage{url}%
\usepackage{hyperref}%
\hypersetup{
  colorlinks=false,
  linkbordercolor=gray,
  citebordercolor=gray,
  urlbordercolor=gray}%

\usepackage{booktabs}
\usepackage{geometry}
\usepackage{tikz,pgfplots}
\usepackage{color}
\usepackage{float}
\usepackage{enumitem}

\newcommand{\E}{\mathbb{E}}

\DeclareMathOperator{\Cov}{Cov}

\DeclareMathOperator{\vol}{vol}

\DeclareMathOperator{\Span}{Span}
\DeclareMathOperator{\Var}{Var}

\DeclareMathOperator{\Aut}{Aut}

\theoremstyle{plain}
\newtheorem{theorem}{Theorem}[section]
\newtheorem{lemma}[theorem]{Lemma}
\newtheorem{proposition}[theorem]{Proposition}

\theoremstyle{definition}
\newtheorem{definition}[theorem]{Definition}
\newtheorem{example}[theorem]{Example}

\theoremstyle{remark}
\newtheorem{remark}[theorem]{Remark}

\title{Noise sensitivity and noise stability for Markov chains: Existence results}
\author{Malin Palö Forsström}
\date{\today}

\begin{document}

\maketitle

\begin{abstract}
During the past 15 years, several extensions of the concepts noise sensitivity and noise stability, first coined in~\cite{schramm2000}, has been studied. The purpose in this paper is to give definitions of this concepts in the setting of continuous time Markov chains, which then unifies many of the previously considered generalizations. In addition, a considerable amount of time is spent on proving the existence of sequences of noise stable and nondegenerate functions with respect to various classes of Markov chains, a problem which interestingly will appear to have close connections to the so called localization of eigenvectors, a problem which in the setting of random graphs has recently been given a lot of attention.
\end{abstract}

\vspace{4em}
\setcounter{secnumdepth}{1}
\setcounter{tocdepth}{1}
\tableofcontents

\section{Introduction}

In~\cite{schramm2000}, Benjamini, Kalai and Schramm coined the term noise sensitivity, looking at how likely the occurance of events where to differ at the starting point and the ending point of sequences of continuous time  random walks on Hamming cubes. Since this paper was published, several extensions of this model, as well as similar definitions in slightly different settings, have been studied, including changing the random walk into an exclusion processes on the Hamming cube \cite{bgs2013} or into Brownian motion on \( \mathbb{R}^n\)  \cite{kms2012b,kms2012a}, as well as trying to understand how the definitions can be applied in the context of functions defined on the leaves of binary trees~\cite{st1999}. Several of the results for the Hamming cube case, which are proven using the theory developed from this research, have also been extended to other settings, such as to the symmetric group or to slices of Hamming cubes~\cite{eff2012, eff2013b, eff2013a, f2014}.

Our main goal of this paper will be to propose a definition of noise sensitivity for general Markov chains and to show that this definition preserves many of the properties from the original setting. 
Interestingly, we will see that some of the questions that arise will have connections to the so called localization of eigenvectors studied recently in eg.~\cite{ab2014},~\cite{dll2011},~\cite{dp2012} and~\cite{esy2009}.

Throughout this paper, we will  be concerned with sequences \( ( X^{(n)})_{n \geq 1} \) of reversible and irreducible continuous time Markov chains \( X^{(n)} \). For each \(n \geq 1\), let \( S^{(n)} \) be the state space, \( Q_n = (q_{ij}^{(n)})_{i,j \in S^{(n)}}\) be the generator and \( \pi_n \) be the stationary distribution of \( X^{(n)} \). We write \( X_t^{(n)} \) to denote the position of \( X^{(n)} \) at time \( t \in \mathbb{R}_+ \), and will always assume that \( X_0^{(n)} \) has been choosen according to \( \pi_n \).

Next, for all \( n \geq 1 \) and \( t \geq 0 \), let \( \smash{H_t^{(n)}}  \)  denote the continuous time Markov semigroup for the Markov chain given by 
\[
H_t^{(n)} = \exp(tQ_n) .
\]
In other words, \( H_t^{(n)} \) operates on a function \( f \) with domain \( S^{(n)} \) by
\[
H_t^{(n)} f(w) = \E[f(X_t^{(n)}) \mid X_0^{(n)} = w].
\]

For functions \( f \) and \( g  \) with domain \( S^{(n)} \), we will use the inner product 
\[
 \langle f , g \rangle = \langle f,g \rangle_{\pi_n}= \E [ f(w)g(w) ].
\]
As \( X^{(n)} \) is assumed to be reversible and irreducible, we can find a set, \( \{ \psi_ j^{(n)}\}_j \) of eigenvectors to \( { -Q_n} \), with corresponding eigenvalues 
\begin{equation}
0 = \lambda_0^{(n)} < \lambda_1^{(n)} \leq \lambda_2^{(n)} \leq \ldots \leq \lambda_{|S^{(n)}|-1}^{(n)}
\label{equation: positive eigenvalues}
\end{equation}
such that  \( \{ \psi_ j^{(n)}\}_j \) is an orthonormal basis with respect to \( \langle \cdot, \cdot \rangle \) for the space of real valued functions on \( S^{(n)} \).
The eigenvectors \( \{ \psi_ j^{(n)}\}_j \) will also be eigenvectors to \( H_t^{(n)} \) with corresponding eigenvalues \( \{ e^{-\lambda_j^{(n)} t} \}_j \).
Since the set \( \{ \psi_ j^{(n)}\}_j \) is an orthonormal basis, for any \( f \colon S^{(n)} \to \mathbb{R} \) we can write
\[
f(w) = \sum_{i=0}^{|S^{(n)}|-1} \langle f, \psi_i^{(n)} \rangle \, \psi_i^{(n)} (w).
\]
To simplify notations, we will write  \( \hat f(j) \) instead of \( \langle f, \psi_j^{(n)} \rangle \).  Note that with this notation, for any function \( f \colon S^{(n)} \to \mathbb{R} \), 
\[
 \E[f] = \E [f \cdot 1] = \langle f, 1 \rangle = \langle f, \psi_0^{(n)} \rangle = \hat f(0) 
\]
and
\begin{equation*}
\begin{split}
\Var(f_n) &= \E[f\cdot f] - \E[f]^2 = \langle f, f \rangle - \hat f(0)^2 
\\&= \left\langle \sum_{i=0}^{|S^{(n)}|-1}  \hat f(i) \psi_i^{(n)} , \sum_{j=0}^{|S^{(n)}|-1} \hat f(j) \psi_j^{(n)} \right\rangle - \hat f(0)^2 
\\&=  \sum_{i=0}^{|S^{(n)}|-1} \sum_{j=0}^{|S^{(n)}|-1}  \hat f(i) \hat f(j) \left\langle \psi_i^{(n)}, \psi_j^{(n)} \right\rangle - \hat f(0)^2
\\&= \sum_{i=0}^{|S^{(n)}|-1} \hat f(i)^2 - \hat f(0)^2= \sum_{i=1}^{|S^{(n)}|-1} \hat f(i)^2
\end{split}
\end{equation*}

The smallest nonzero eigenvalue, \( \lambda_1^{(n)} \), is called the spectral gap of the Markov chain \( X^{(n)} \), and its inverse, \( t_{rel}^{(n)} \coloneqq 1/\lambda_1^{(n)} \) is called the relaxation time. Another characterization of the spectral gap which will be useful for us is
\begin{equation}
\lambda_1^{(n)}
= 
\min_{f \colon \E[f]=0, f \not \equiv 0} \frac{\langle -Q^{(n)}f, f \rangle}{\langle f, f \rangle},
\label{equation: rayleigh quotient}
\end{equation}
where  the minimum is attained by the corresponding eigenvector \( \psi_1^{(n)} \). The right  hand side of~\eqref{equation: rayleigh quotient} is called the Rayleigh quotient of \( -Q^{(n)} \).

We will now define the concepts with which we will be concerned in the rest of these notes.

\begin{definition}
Let \( (X^{(n)})_{n \geq 1} \) be a sequence of reversible and irreducible continuous time Markov chains, with state spaces \( (S^{(n)})_{n \geq 1} \) and stationary distributions \( (\pi_n)_{n \geq 1} \), and let \( t_{rel}^{(n)} \) be the relaxation time of the \( n\)th Markov chain. A sequence \( (f_n )_{n \geq 1} \)  of Boolean functions, \( f_n \colon S^{(n)} \to \{ 0,1 \} \), is said to be \emph{noise sensitive} with respect to \( (X^{(n)})_{n \geq 1} \) , if for all \( \alpha > 0 \)
\begin{equation}
\lim_{n \to \infty} \Cov \left(f_n(X_0^{(n)}), f_n(X_{\alpha t_{rel}^{(n)}}^{(n)})\right) = 0.
\label{def: possible definition rt}
\end{equation}
\label{definition: noise sensitivity}
\end{definition}

Note that as \(\E \left[ f_n(X_0^{(n)}) \right] = \E \left[ f_n(X_{\alpha t_{rel}^{(n)}}^{(n)}) \right]\), we have that
\[
\Cov(f_n(X_0^{(n)}), f_n(X_{\alpha t_{rel}^{(n)}}^{(n)})
= 
 \E \left[ f_n(X_0^{(n)}) f_n(X_{\alpha  t_{rel}^{(n)}}^{(n)}) \right] - \E \left[ f_n(X_0^{(n)}) \right]^2 .
\]

In addition to the definition for noise sensitivity given above, we will use the following definition of \emph{noise stability}, which captures the opposite behaviour.
\begin{definition}
Let \( (X^{(n)})_{n \geq 1} \) be a sequence of reversible and irreducible continuous time Markov chains, with state spaces \( (S^{(n)})_{n \geq 1} \) and stationary distributions \( (\pi_n)_{n \geq 1} \), and let \( t_{rel}^{(n)} \) be the relaxation time of the \( n\)th Markov chain. A sequence \( (f_n )_{n \geq 1} \)   of Boolean functions, \( f_n \colon S^{(n)} \to \{ 0,1 \} \), is said to be \emph{noise stable} with respect to \( (X^{(n)})_{n \geq 1} \) if 
\begin{equation}
\lim_{\alpha \to 0} \sup_n P \left( f_n(X_0^{(n)}) \not = f_n(X_{\alpha t_{rel}^{(n)}}^{(n)}) \right) = 0.
\end{equation}
\end{definition}

As we will see later, if \( (X^{(n)})_{n \geq 1} \) is a sequence of reversible and irreducible continuous time Markov chains and  \( (f_n)_{n\geq 1} \) is a sequence of Boolean functions with domain \( S^{(n)} \)  such that \( \lim_{n\to \infty} \Var(f_n) = 0 \), then \( (f_n)_{n\geq 1} \) will be both noise stable and noise sensitive with respect to \( (X^{(n)})_{n \geq 1} \). For this reason, we are only interested in sequences of Boolean functions with
\[
\lim_{n\to \infty} \Var(f_n) > 0.
\]
If this is satisfied for  a sequence \( (f_n)_{n\geq 1} \) of Boolean functions, we say that \( (f_n)_{n\geq 1} \) is \emph{nondegenerate}.

\begin{remark}
In~\cite{schramm2000}, the concept of noise sensitivity was defined as follows. Given \( {X_0^{(n)} = (X_0^{(n)}(1), \ldots, X_0^{(n)}(n)) \in \{ 0,1 \}^n }\), let  \( {\tilde X_\alpha^{(n)} = (\tilde X_\alpha^{(n)}(1), \ldots, \tilde X_\alpha^{(n)}(n)) \in \{ 0,1 \}^n} \) be a random perturbation of \( X_0^{(n)} \), i.e. for every \( j \in \{1, \ldots , n\} \) independently, set \( {\tilde X_\alpha^{(n)}(j) = X_0^{(n)}(j)}\) with probability \( 1-\alpha \), and \( \tilde X_\alpha^{(n)}(j)) = 1-X_0^{(n)}(j)\)  else. A sequence \( (f_n)_{n \geq 1 } \), \( f_n\colon \{ 0,1 \}^n \to 0\), was said to be \emph{asymptotically noise sensitive} if for all \( \alpha > 0 \), 
\[
\lim_{n \to \infty} \Cov(f_n(X_0^{(n)}), f_n(\tilde X_{\alpha}^{(n)})= 0.
\]

Now consider the continuous time Markov chain \( X^{(n)} \) on \( \{ 0,1 \}^n \) which for each coordinate \( i \), at times that are exponentially distributed with parameter \( 1 \), rerandomizes the value at \( i \) by setting the value to 1 with probability \( 1/2 \) and 0 with probability \( 1/2 \). This Markov chain has relaxation time \( 1 \), and for any specific coordinate \( i \), the probability that the value at \( i \) is not the same at time \( 0 \) as at time \( \alpha t_{rel}^{(n)} = \alpha \) is  \( (1-e^{-\alpha })/2 \).
This implies that the law of \( \smash{X_{\alpha t_{rel}^{(n)}  }^{(n)}} \) is the same as the law of \( \tilde X^{(n)}_{(1-e^{-\alpha})/2} \), why Definition~\ref{definition: noise sensitivity} in this special case is equivalent with the definition of asymptotic noise sensitivity given in~\cite{schramm2000}.

Comparing Definition~\ref{definition: noise sensitivity} with the definition of \emph{complete graph noise sensitivity} given  in~\cite{bgs2013}, the two definitions coincide if we consider the continuous time Markov chain corresponding to a  random walk on the union of the graphs \( G_n \) with vertices \( \{ w \in \{ 0,1 \}^n \colon \| w \| = k  \} \) and an edge between two vertices \( u \) and \( w \) iff \( \| u \| = \| w \| \) and \( \| u - v \| = 2 \), and choose \( \pi_n \) to be the uniform measure on \( \{ 0,1 \}^n \). However, note that as these graphs are not connected, the corresponding Markov chains are not irreducible. We will therefore not consider this model directly in this paper, even though some of our results applies to this model as well.
\end{remark}

Using the eigenvalues of the generator, we will now give another characterization of noise sensitivity, which generalizes the first part of Theorem 1.9 in~\cite{schramm2000}.

\begin{proposition}
A sequence of Boolean functions \( ( f_n )_{n \geq 1} \), \( f_n \colon S^{(n)} \to \{ 0,1 \} \), is noise sensitive with respect to \( (X^{(n)})_{n \geq 1} \)  if and only if for all \( k > 0 \), 
\begin{equation}
\lim_{n \to \infty} \sum_{i\colon \lambda_1^{(n)} \leq \lambda_i^{(n)} < k\lambda_1^{(n)}}  \hat f_n(i)^2 = 0.
\label{eq: epsilon dependency}
\end{equation}
\label{proposition: noise sensitivity fourier}
\end{proposition}

\begin{proof}
Fix \( \alpha > 0 \). Then
\begin{equation*}
\begin{split}
\E \left[ f_n(X_0^{(n)})f_n(X_{\alpha t_{rel}^{(n)}}^{(n)}) \right]
&=
\E \left[f_n(X_0^{(n)})H^{(n)}_{\alpha t_{rel}^{(n)}} f_n(X_0^{(n)}) \right]
\\&=
\E \left[\sum_{i=0}^{|S^{(n)}|-1}\hat f_n(i) \psi_i^{(n)} (X_0^{(n)}) \sum_{j=0}^{|S^{(n)}|-1} \hat f_n(j) H^{(n)}_{\alpha t_{rel}^{(n)}} \psi^{(n)}_j (X_0^{(n)}) \right]
\\&=
\E \left[\sum_{i=0}^{|S^{(n)}|-1}\hat f_n(i) \psi_i^{(n)} (X_0^{(n)}) \sum_{j=0}^{|S^{(n)}|-1} \hat f_n(j) e^{ - \alpha t_{rel}^{(n)} \cdot \lambda_j^{(n)} }\psi^{(n)}_j (X_0^{(n)}) \right]
\\&=
\sum_{i,j} e^{{-\alpha t_{rel}^{(n)}} \cdot \lambda_j^{(n)} } \hat f_n(i) \hat f_n(j)  \E \left[\psi_i^{(n)} (X_0^{(n)}) \psi^{(n)}_j (X_0^{(n)}) \right]
\\&=
\sum_{i=0}^{\mathclap{|S^{(n)}|-1}} e^{{-\alpha \lambda_i^{(n)}/\lambda_1^{(n)}} } \hat f_n(i)^2 .
\end{split}
\end{equation*}
Here the last equality follows from the fact that   \( \{ \psi_i^{(n)} \}_{i = 0}^{|S^{(n)}|-1} \) is an orthonormal set, together with the definition of the relaxation time.
As \(\E \left[ f_n (w) \right]^2 = \hat f_n(0)^2\), it follows that 
\begin{equation}
\begin{split}
\Cov( f_n(X_0^{(n)}), f_n(X_{\alpha t_{rel}^{(n)}}^{(n)}))  
&= 
\E \left[ f_n(X_0^{(n)})f_n(X_{\alpha t_{rel}^{(n)}}^{(n)}) \right] -\E \left[ f_n (X_0^{(n)}) \right]^2  
\\&= 
\sum_{j=1}^{\mathclap{|S^{(n)}|-1}} e^{{-\alpha \lambda_j^{(n)}/\lambda_1^{(n)}} } \hat f_n(i)^2 .
\label{equation: spectral formulation of noise sensitivity}
\end{split}
\end{equation}
For any \( \alpha > 0 \), it is easy to see that the left hand side of~\eqref{equation: spectral formulation of noise sensitivity} tends to zero as \( n \to \infty \) if and only if~\eqref{eq: epsilon dependency} holds. From this the desired conclusion follows.
\end{proof}

\begin{remark}
By the last lines of the proof above it follows that  if a sequence of functions satisfies~\eqref{def: possible definition rt} for one \( \alpha > 0 \), then it does so for all \( \alpha > 0 \), i.e. the proof of Proposition~\ref{proposition: noise sensitivity fourier} in fact shows that a sequence of Boolean functions \( (f_n)_{n \geq 1} \) is noise sensitive with respect to \( (X^{(n)})_{n \geq 1} \) if and only if
\begin{equation*}
\lim_{n \to \infty}\Cov( f_n(X_0^{(n)}), f_n(X_{ t_{rel}^{(n)}}^{(n)}))  = 0.
\end{equation*}
\label{remark: epsilon dep}
\end{remark}

We can easily obtain a similar characterization of noise stability,  which generalizes the second part of Theorem 1.9 in~\cite{schramm2000}.

\begin{proposition}
Let \( (X^{(n)} )_{n\geq 1 } \) be a sequence of   reversible and irreducible continuous time Markov chains. A sequence of Boolean functions \( (f_n)_{n \geq 1} \), \( f_n \colon S^{(n)} \to \{ 0,1 \} \), is noise stable with respect to \( (X_n)_{n\geq 1} \) if and only if for all \( \delta > 0 \) there is \( k \in \mathbb{N} \) such that
\begin{equation}
\sup_{n} \sum_{i \colon \lambda_i^{(n)} \geq k\lambda_1^{(n)}} \hat f_{n}(i)^2 < \delta.
\label{equation: noise stability fourier}
\end{equation}
\label{proposition: noise stability fourier}
\end{proposition}

\begin{remark}
Note that the proposition above directly implies that if \( ( X^{(n)} )_{n\geq 1 } \) is a  sequence of Markov chains corresponding to random walks on a family of expander graphs, i.e. a sequence of graphs where the correspondinging sequence of spectral gaps is bounded from below, then all sequences of Boolean functions on \( (S^{(n)})_{n \geq 1} \) will be noise stable with respect to \(( X_n)_{n \geq 1 } \).
\label{remark: expander graphs}
\end{remark}

\begin{proof}
First note that since \( f_n \) is Boolean, we have that
\begin{equation*}
\begin{split}
P\left( f_n(X^{(n)}_{\alpha t_{rel}^{(n)}}) \not = f_n(X^{(n)}_0)\right)
&= 
\E \left[ f_n(X^{(n)}_{\alpha t_{rel}^{(n)}}) (1-f_n(X^{(n)}_0)) \right] + \E \left[ f_n(X^{(n)}_0) (1-f_n(X^{(n)}_{\alpha t_{rel}^{(n)}}))\right]
\\&= 2 \left( \E \left[ f_n(X^{(n)}_0) \right]  - \E \left[f_n(X^{(n)}_0)  f_n(X^{(n)}_{\alpha t_{rel}^{(n)}}) \right] \right) .
\end{split}
\end{equation*}
Using this, as well as the the proof of the previous proposition, we obtain
\begin{equation*}
\begin{split}
P \left( f_n(X_0^{(n)}) \not = f_n(X_{\alpha t_{rel}^{(n)}}^{(n)}) \right)
&=
2 \left(\E \left[ f_n(X^{(n)}_0) \right]  -  \sum_{i=0}^{|S^{(n)}|-1} e^{-\alpha \lambda_i^{(n)}/\lambda_1^{(n)}} \hat f_n(i)^2\right)
\\&=
2 \left(\E \left[ f_n(X^{(n)}_0)^2 \right]  -  \sum_{i=0}^{|S^{(n)}|-1} e^{-\alpha \lambda_i^{(n)}/\lambda_1^{(n)}} \hat f_n(i)^2\right)
\\&=
2 \left(\sum_{i=0}^{|S^{(n)}|-1} \hat f(i)^2  -  \sum_{i=0}^{|S^{(n)}|-1} e^{-\alpha \lambda_i^{(n)}/\lambda_1^{(n)}} \hat f_n(i)^2\right)
\\&=
2 \sum_{i=0}^{|S^{(n)}|-1}\left(1-e^{-\alpha \lambda_i^{(n)}/\lambda_1^{(n)}} \right) \hat f_n(i)^2.
\end{split}
\end{equation*}

For the \emph{if} direction of the proof, suppose that there for any any \( \delta > 0 \) is \( k_\delta \geq 1\) such that
\begin{equation*}
\sup_{n} \sum_{i \colon \lambda_i^{(n)} \geq k_\delta \lambda_1^{(n)}} \hat f_{n}(i)^2 < \delta.
\end{equation*}
Then for all \( \delta > 0 \),
\begin{equation*}
\begin{split}
\lim_{\alpha \to 0 } \sup_n P\left( f_n(X^{(n)}_{\alpha t_{rel}^{(n)}}) \not = f_n(X^{(n)})\right)
&=
2\lim_{\alpha \to 0 } \sup_n  \sum_{i=0}^{|S^{(n)}|-1} \left(1-e^{-\alpha \lambda_i^{(n)}/\lambda_1^{(n)}} \right) \hat f_n(i)^2
\\&\leq
2\delta + 2\lim_{\alpha \to 0 } \sup_n  \sum_{i \colon \lambda_i^{(n)} < k_\delta \lambda_1^{(n)}} \left(1-e^{-\alpha k_\delta \lambda_1^{(n)}/\lambda_1^{(n)}} \right) \hat f_n(i)^2
\\&=
2\delta +2\lim_{\alpha \to 0 }   \sup_n  \sum_{i \colon \lambda_i^{(n)} < k_\delta \lambda_1^{(n)}} \left(1-e^{-\alpha k_\delta } \right) \hat f_n(i)^2
\\&\leq
2\delta +2\lim_{\alpha \to 0 }  \left(1-e^{-\alpha k_\delta } \right) 
\\&= 2\delta.
\end{split}
\end{equation*}
As \( \delta \) can be chosen to be arbitrarily small, this implies that \( (f_n)_{n \geq 1} \) is noise stable with respect to \( (X_n)_{n\geq 1} \).

For the \emph{only if} direction, suppose that there is \( \delta > 0 \) such that for all \( k \geq 1\),
\begin{equation*}
\sup_{n} \sum_{i \colon \lambda_i^{(n)} \geq k\lambda_1^{(n)}} \hat f_{n}(i)^2 \geq \delta
\end{equation*}
for all \( k>0\). Then in particular, this is true for \( k = \alpha^{-1} \). This implies that
\begin{equation*}
\begin{split}
\lim_{\alpha \to 0 } \sup_{n} P \left( f_n(X_0^{(n)}) \not = f_n(X_{\alpha t_{rel}^{(n)}}^{(n)}) \right)
&=
2 \lim_{\alpha \to 0 } \sup_{n}    \sum_{i=0}^{|S^{(n)}|-1} \left(1-e^{-\alpha \lambda_i^{(n)}/\lambda_1^{(n)}} \right) \hat f_n(i)^2
\\&\geq
2 \lim_{\alpha \to 0 } \sup_{n}    \sum_{i \colon \lambda_i^{(n)} \geq k \lambda_1^{(n)}} \left(1-e^{-\alpha \lambda_i^{(n)}/\lambda_1^{(n)}} \right) \hat f_n(i)^2
\\&\geq
2 \lim_{\alpha \to 0 } \sup_{n}    \sum_{i \colon \lambda_i^{(n)} \geq k \lambda_1^{(n)}} \left(1-e^{-\alpha k} \right) \hat f_n(i)^2
\\&=
2 \lim_{\alpha \to 0 } (1-e^{-1}) \delta.
\end{split}
\end{equation*}
In particular, \( (f_n)_{n \geq 1} \) cannot be noise stable.
\end{proof}

\begin{remark}
The proof of Proposition~\ref{proposition: noise sensitivity fourier} shows that Definition~\ref{definition: noise sensitivity} is sharp in the following sense. Let \( ( T_n)_{n \geq 1} \) be any sequence such that   \(\lim_{n \to \infty} T_n / t_{rel}^{(n)} = \infty \). Then for any sequence \( (f_n)_{n \geq 1}\) of Boolean functions with domain \( S^{(n)} \), and any \( \alpha > 0 \), 
\begin{equation}
\lim_{n \to \infty} \Cov(f_n(X_0^{(n)}), f_n(X_{\alpha T_n}^{(n)}) = 0
\end{equation}
as we, by the same method as the one used in this proof, can show that
\begin{equation*}
\begin{split}
\Cov(f_n(X_0^{(n)}), f_n(X_{\alpha T_n}^{(n)}) 
&= \sum_{i=1}^{\mathclap{|S^{(n)}|-1}} e^{{-\alpha \lambda_i^{(n)}T_n} } \hat f_n(i)^2
\leq  
e^{{-\alpha \lambda_1^{(n)}T_n} } \sum_{i=1}^{\mathclap{|S^{(n)}|-1}} \hat f_n(i)^2 
\\&\leq 
e^{{-\alpha \lambda_1^{(n)}T_n} } \langle f, f  \rangle 
=
e^{{-\alpha \lambda_1^{(n)}T_n} } \E [f^2]
\leq 
e^{{-\alpha \lambda_1^{(n)}T_n} }
\end{split}
\end{equation*}
and this tends to zero as \( n \to \infty \). The sequence of relaxation times is thus the largest sequence of times we can look at and still possibly obtain a nontrivial definition of noise sensitivity. By a similar argument, we find that for any  \( \alpha > 0 \) and sequence \( ( T_n)_{n \geq 1} \)  such that \( {\lim_{n\to \infty} T_n / t_{rel}^{(n)} = \infty} \), then for all sequences \( (f_n)_{n\geq 1} \) of Boolean functions, 
\begin{equation*}
\begin{split}
P(f_n(X_0^{(n)}) \not = f_n(X_{\alpha T_n}^{(n)})) 
&= 
2 \sum_{i=0}^{|S^{(n)}|-1} \left(1-e^{-\alpha \lambda_i^{(n)}T_n} \right) \hat f_n(i)^2 
\\&\geq 
2 \left(1-e^{-\alpha \lambda_1^{(n)}T_n} \right)\sum_{i \in S^{(n)} \colon i \geq 1}  \hat f_n(i)^2  
\\&\geq 
2 \Var(f_n)
\end{split}
\end{equation*}
for all large enough \( n \). This implies that if we replaced \( t_{rel}^{(n)} \) by any sequence \( (T_n)_{n \geq 1} \) with \( \lim_{n\to \infty} T_n / t_{rel}^{(n)} = \infty\) in the definition of noise stability, then no {nondegenerate} sequence of Boolean functions  would be noise stable with respect to \( (X^{(n)})_{n \geq 1} \).
\end{remark}

\begin{remark}
From the fact that \( \Var(f_n) = \sum_{i = 1}^{|S^{n}|-1} \hat f(i)^2 \), using  Proposition~\ref{proposition: noise sensitivity fourier} and Proposition~\ref{proposition: noise stability fourier} we see that degenerate  sequence \( (f_n)_{n\geq 1} \) of Boolean functions with domains \( (S^{(n)})_{n\geq 1} \) will be both noise stable and noise sensitive  with respect to any sequence of reversible and irreducible Markov chains with state spaces \( (S^{(n)})_{n\geq 1} \).
\end{remark}

In the remainder of these notes, we will often be concerned with sequences of Markov chains \( (X^{(n)})_{n \geq 1} \) corresponding to random walks on sequences of connected graphs \( (G_n)_{n \geq 1} \). By a \emph{random walk} on a graph \( G_n \) we mean a continuous time Markov chain with state space \( V(G_n) \) and generator \( Q_n = (q_{vw}^{(n)})_{v,w \in V(G_n)} \) with
\[ 
q^{(n)}_{vw} = 
  \begin{cases} 
    1/\deg(v) &\textnormal{when \( v \) and \( w \) are neighbours} 
    \cr -1 &\textnormal{if \( v=w \)} \cr 
    0 &\textnormal{else}
  \end{cases}
\]
where \( \deg(v) \) is the degree of the vertex \( v \).

Whenever we talk about a sequence of graphs \( (G_n)_{n \geq 1 }\), we will assume that  \( |V(G_n)| \to \infty \) as \( n \to \infty \) and that \( G_n \) is connected for every \( n \). For a graph \( G \), we will use \( \vol(G) \) to denote twice the number of (undirected) edges in the graph. For example, for the complete graph on \( n \) vertices, \( K_n \), we have that \( \vol(K_n) = n(n-1) \).

\section{The existence of noise sensitive functions}

As our definitions of noise sensitivity and noise stability coincide with the definitions given in~\cite{schramm2000}, we already know that both noise stable and noise sensitive nondegenerate sequences of Boolean functions exist  with respect to \( (X^{(n)})_{n\geq 1} \) when \( X^{(n)} \) is a random walk on a {\(n\)-dimensional} Hamming cube. A natural question is if such sequences will exist in general, or if the Hamming cube is a very special case. Recall that by Remark~\ref{remark: expander graphs}, there are sequences of Markov chains for which no such sequences exists.  In this section we will focus on the question about the existence of sequences of noise sensitive functions, and try to give criteria on the Markov chain to guarentee just that.

We begin this section by considering an example of a family of expander graphs, and give a more concrete explanation of why no noise sensitive sequences of functions can exist on these graphs.

\begin{example}
Fix \( k \geq 1 \) and for each \( n \geq k \), let \( X^{(n)} \) be the continuous time Markov chain with state space \( \smash{\binom{[n]}{k}} \coloneqq \{ (w_1, \ldots, w_n) \in \{0,1\}^n \colon w_1 + \cdots + w_n = k \}\)  which evolves as follows. 

At times which are exponentially distributed with parameter 1, two indices \( i,j \in [n] \) are chosen independently at random, and the digits at these positions in the current string are switched. Equivalently,  \( X^{(n)} \) is an exclusion process on the binary strings of length \( n \) and Hamming weight \( k \). For any fixed \(k \), as \( n \to \infty \) this Markov chain has relaxation time of order \(1\) (see e.g.~\cite{ds1987}). As a consequence of this, as \( \alpha \to 0 \), at time \( \alpha \cdot 1 \) the probability is very high that none of the balls have moved at all, and thus all sequences of functions are noise stable with respect to this sequence of Markov chains.
\end{example}

\begin{remark}
The simplest of the models in the previous example is given by choosing \( k = 1 \), in which case we obtain a random walk on the complete graph on \( n \) vertices.
\end{remark}

The moral of the previous example is that if for some sequence of random walks on a sequence of graphs, the relaxation time 
is such that \(t_{rel}^{(n)} \) is bounded from above when \( n \) tends to infinity, then all functions will be noise stable, the reason being that with high probability, the random walker on the \(n\)th graph never moves at all between time 0 and time \( \alpha t_{rel}^{(n)} \). 
Proposition~\ref{proposition: noise stability fourier} shows that another condition which guarantees that all functions are noise stable is that \( \smash{\lambda_{|S^{(n)}|-1}^{(n)} }/ \lambda_1^{(n)} = \mathcal{O}(1) \). However, as for random walks on graphs, \( \lambda_{|S^{n-1}|} \leq 2 \), these conditions are equivalent.

In the rest this section, our goal will be to prove the following result, which shows that when with high  probability,  the position of the random walker at the relaxation time is  not the same as the position of the random walker at time zero, there will be at least one nondegenerate sequence of  functions which is noise sensitive.

\begin{proposition}
Let \( ( X^{(n)} )_{n \geq 1} \) be a sequence of reversible and irreducible continuous time Markov chains. Then  if 
\begin{equation}
\lim_{n \to \infty} \left\{ P( X^{(n)}_0 = X^{(n)}_{ t_{rel}^{(n)}}) -  \sum_{i } \pi_n(i)^2  \right\} = 0
\label{equation: condition for existence of sensitiive function}
\end{equation}
there is at least one nondegenerate sequence of functions \( ( f_n )_{n \geq 1}  \) which is noise sensitive with respect to \( (X^{(n)})_{n\geq 1 } \).

\label{proposition: noise sensitive functions}
\end{proposition}

\begin{remark}
Note that~\eqref{equation: condition for existence of sensitiive function} is equivalent to that
\begin{equation}
\lim_{n \to \infty} \left\{ \sum_{w \in S^{(n)}  } \pi_n(w) \cdot \left( P(  X^{(n)}_{ t_{rel}^{(n)}} = w \mid X^{(n)}_0 = w) -  P( X^{(n)}_{ t_{rel}^{(n)}} = w) \right)   \right\} =0
\end{equation}
i.e.~\eqref{equation: condition for existence of sensitiive function} is a measure of how different the probability of being at \( X_0^{(n)} \) at time \( t_{rel}^{(n)} \) is from being in \( X_0^{(n)}  \) when in the stationary distribution. Also,~\eqref{equation: condition for existence of sensitiive function} can be rewritten as
\begin{equation}
\lim_{n \to \infty} \left\{ \sum_{w \in S^{(n)} } \Cov \Bigl( \mathbf{1}_w(X_0^{(n)}), \mathbf{1}_w(X^{(n)}_{t_{rel}^{(n)}}) \Bigl) \right\} = 0
\end{equation}
i.e. as the sum of the noise sensitivity of the indicator functions of all \( w \in S^{(n)} \).
\end{remark}

Even though in general, it might be hard to check whether~\eqref{equation: condition for existence of sensitiive function} holds, Proposition~\ref{proposition: noise sensitive functions} might be useful in some special cases. It is e.g. relatively simple to show that~\eqref{equation: condition for existence of sensitiive function} holds for graphs whose minimum degree tends to infinity, such as sequences of hypercubes. 
On the other hand, the following example, due to Johan Jonasson, shows that there are sequences of graphs on which~\eqref{equation: condition for existence of sensitiive function} do not hold, even though the corresponding sequence of relaxation times is unbounded.

\begin{example}
Let \( G_n\) be the graph obtained by joining \( 2n \) stars with \( n^2 \) outer vertices byadding all possible edges between their centers (see figure~\ref{figure: Johans example}). This relaxation time of the random walk on this graph can be at most of the same order as the expected time until one of the edges between the centers is used, which is of order \(n\). As the bottleneck ratio can be at most \( 1/n \), it follows that the relaxation time \( t_{rel}^{(n)} \) in fact equals \( n \).

Now pick \( \varepsilon>0 \) to be very small. Then with probability close to 0.5, given that the random walk started in one of the vertices of the inner complete graph, \( x \), at time \( \varepsilon t_{rel}^{(n)} \), the random walk is at the same position as where it started. By a similar argument, it follows that given that the random walk started in one of the outer vertices, \( y \), at time \( \varepsilon t_{rel}^{(n)} \) the probability is close to \( 1/2n^2 \) that the random walk is in \( y \). Using this, as well as standard results for random walk on discrete graphs, it follows that~\eqref{equation: condition for existence of sensitiive function} do not hold. However, as the mixing time of each individual star is of order one, it is relatiively easy to see that any sequence of functions with \( \lim_{n \to \infty} \Var_i(\E[f_n(X_0^{(n)} \mid X_0^{(n)} \textnormal{ in star number } i]) = 0 \) will be noise sensitive.

\begin{figure}[H]
\begin{tikzpicture}[scale = 1.5]
  \def\N{5}
  \def\a{6}

  \foreach \i in {0,...,\N}{
    \foreach \j in {\i,...,\N}{
    \draw ({cos((\i+0.5)*360/(\N+1))},{sin((\i+0.5)*360/(\N+1))}) -- ({cos((\j+0.5)*360/(\N+1))},{sin((\j+0.5)*360/(\N+1))});
    };
  };

  \foreach \i in {0,...,\N}{
    \foreach \j in {1,...,\a}{
      \draw ({cos((\i+0.5)*360/(\N+1))},{sin((\i+0.5)*360/(\N+1))})  --  
               ({1.7*cos((\a*(\i+0.5)+0.5*(\j-0.5*\a-0.5))*360/(\a*(\N+1)))},{1.7*sin((\a*(\i+0.5)+0.5*(\j-0.5*\a-0.5))*360/(\a*(\N+1)))});
      \draw[fill = black] ({1.7*cos((\a*(\i+0.5)+0.5*(\j-0.5*\a-0.5))*360/(\a*(\N+1)))},{1.7*sin((\a*(\i+0.5)+0.5*(\j-0.5*\a-0.5))*360/(\a*(\N+1)))}) circle (1pt);
    };
  };

  \foreach \i in {0,...,\N}{
    \draw[fill = white] ({cos((\i+0.5)*360/(\N+1))},{sin((\i+0.5)*360/(\N+1))})  circle (2pt);
  };

 \end{tikzpicture}

\caption{The figure above shows the graph \( G_n\), which is obtained by joining \( n \) stars with \( n \) leaves by adding all possible edges between their centers.}
\label{figure: Johans example}
\end{figure}
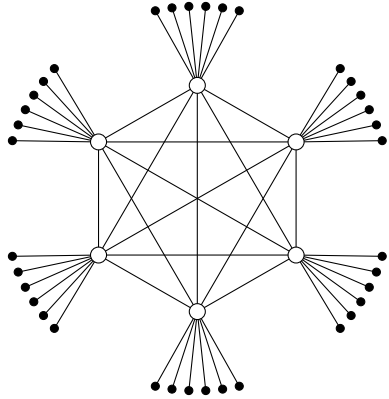

\end{example}

We now give a proof of Proposition~\ref{proposition: noise sensitive functions}.
\begin{proof}[Proof of Proposition~\ref{proposition: noise sensitive functions}]

Fix some \(n \geq 1 \). Let \( (m_n)_{n\geq 1} \) be a sequence of  integers s.t. \( \lim_{n \to \infty } m_n/|S^{(n)}| \in (0,1) \). Define \( f_n \) by choosing exactly \( m_n  \) of the states in \( S^{(n)} \) independently at random, and set \( f_n \) to be one on those vertices and \( 0 \) everywhere else. Then 
\[ 
\lim_{n \to \infty} \Var [f_n(w)] =  \lim_{n \to \infty} m_n(|S^{(n)}|-m_n)/|S^{(n)}|^2 \in (0,1) 
\]
so that any sequence \( ( f_n)_{n \geq 1 } \) made by picking a sequence of functions in this way is nondegenerate.

Define the random vaiable \( Y_m\) in terms of \( f_n\) be 
\[
Y_n = \Cov \left( f_n(X^{(n)}_0), f_n(X^{(n)}_{ t_{rel}^{(n)}}) \right) = \E \left[ f_n(X^{(n)}_0) f_n(X^{(n)}_{ t_{rel}^{(n)}})\right] - \E \left[ f_n(X^{(n)}_0) \right]^2 .
\]
and let \( v^{(n)}, w^{(n)} \in S^{(n)} \), \( v^{(n)} \not = w^{(n)} \) be fixed.
Then
\begin{equation}
\begin{split}
\E[Y_n] &=  \E \left[ \E \Bigl[ f_n(X^{(n)}_0) f_n(X^{(n)}_{ t_{rel}^{(n)}}) \mid f_n\Bigr] - \E \Bigl[ f_n(X^{(n)}_0) \mid f_n\Bigr]^2 \right] 
\\&=
E_{f_n} \left[ \E \Bigl[ f_n(X^{(n)}_0) f_n(X^{(n)}_{ t_{rel}^{(n)}}) \mid f_n \Bigr] \right] - \E \left[  \E \Bigl[ f_n(X^{(n)}) \mid f_n \Bigr]^2 \right] .
\end{split}
\label{equation: E of Xn}
\end{equation}
Rewriting the first of these two terms, we obtain
\begin{equation*}
\begin{split}
\E \left[ \E \left[ f_n(X^{(n)}_0) f_n(X^{(n)}_{ t_{rel}^{(n)}}) \mid f_n \right]  \right]
=
&\E \left[ \E \Bigl[  f_n(X^{(n)}_0) f_n(X^{(n)}_{ t_{rel}^{(n)}})\mid  X_0^{(n)},\, X_{t_{rel}^{(n)}}^{(n)}\Bigr]  \right]
\\=
&\E \left[ \E_{f_n} \Bigl[ f_n(X^{(n)}_0) f_n(X^{(n)}_{ t_{rel}^{(n)}}) \mid  X_0^{(n)},\, X_{t_{rel}^{(n)}}^{(n)} \Bigr] \mid X^{(n)}_0 = X^{(n)}_{ t_{rel}^{(n)}}\right]   \cdot P( X^{(n)}_0 = X^{(n)}_{ t_{rel}^{(n)}})
\\
+ &\E \left[ \E_{f_n} \Bigl[ f_n(X^{(n)}_0) f_n(X^{(n)}_{ t_{rel}^{(n)}}) \mid  X_0^{(n)},\, X_{t_{rel}^{(n)}}^{(n)} \Bigr] \mid X^{(n)}_0 \not = X^{(n)}_{ t_{rel}^{(n)}} \right] \cdot P( X^{(n)}_0 \not = X^{(n)}_{ t_{rel}^{(n)}}).
\end{split}
\end{equation*}
Here
\begin{equation*}
\begin{split}
\E \left[ \E \Bigl[ f_n(X^{(n)}_0) f_n(X^{(n)}_{ t_{rel}^{(n)}}) \mid  X_0^{(n)},\, X_{t_{rel}^{(n)}}^{(n)} \Bigr] \mid X^{(n)}_0 = X^{(n)}_{ t_{rel}^{(n)}}\right]  
&=
\E \left[ \E \Bigl[ f_n(X^{(n)}_0)^2 \mid  X_0^{(n)}  \Bigr] \mid X^{(n)}_0 = X^{(n)}_{ t_{rel}^{(n)}}\right]  
\\&=
\E \left[ \E \Bigl[ f_n(X^{(n)}_0) \mid  X_0^{(n)}  \Bigr] \mid X^{(n)}_0 = X^{(n)}_{ t_{rel}^{(n)}}\right]  
\\&=
\E \left[ \E \Bigl[ f_n(X^{(n)}_0) \mid  X_0^{(n)}  \Bigr] \right]  
\\&= {m_n}/{|S^{(n)}|}
\end{split}
\end{equation*}
and similarly
\begin{equation*}
\E \left[ \E \Bigl[ f_n(X^{(n)}_0) f_n(X^{(n)}_{ t_{rel}^{(n)}}) \mid  X_0^{(n)},\, X_{t_{rel}^{(n)}}^{(n)}  \Bigr] \mid X^{(n)}_0 \not = X^{(n)}_{ t_{rel}^{(n)}} \right]
=
 \left.  {\binom{m_n}{2}} \right/ {\binom{|S^{(n)}|}{2}} 
=
 \frac{m_n(m_n-1)}{|S^{(n)}|(|S^{(n)}|-1)}.
\end{equation*}
For the last term in~\eqref{equation: E of Xn}
\begin{equation*}
\begin{split}
\E  [ \E [f_n(X^{(n)}_0) \mid f_n]^2] 
&= \E  \left[ P(f_n(X_0^{(n)}) = 1 \mid f_n)^2\right]
\\&=
\binom{|S^{(n)}|}{m_n}^{-1} \cdot \left( \binom{|S^{(n)}|-1}{m_n-1} \sum_{i \in S^{(n)}}  \!\! \pi_n(i)^2 +  \binom{|S^{(n)}|-2}{m_n-2} \sum_{i,j \in S^{(n)}\colon  i\not=j} \!\!\!\! \pi_n(i)\pi_n(j) \right)
\\&=
\frac{m_n}{|S^{(n)}|} \sum_{i \in S^{(n)}} \pi_n(i)^2 +  \frac{m_n(m_n-1)}{|S^{(n)}|(|S^{(n)}|-1)} \sum_{i,j \in S^{(n)} \colon  i\not=j} \pi_n(i)\pi_n(j) 
\\&=
\left( \frac{m_n}{|S^{(n)}|} -  \frac{m_n(m_n-1)}{|S^{(n)}|(|S^{(n)}|-1)}  \right) \sum_{i } \pi_n(i)^2 +  \frac{m_n(m_n-1)}{|S^{(n)}|(|S^{(n)}|-1)} \sum_{ i,j} \pi_n(i)\pi_n(j) 
\\&=
 \frac{m_n(|S^{(n)}|-m_n)}{|S^{(n)}|(|S^{(n)}|-1)} \sum_{i } \pi_n(i)^2 +  \frac{m_n(m_n-1)}{|S^{(n)}|(|S^{(n)}|-1)} \left( \sum_{ i} \pi_n(i) \right)^2
\\&=
 \frac{m_n(|S^{(n)}|-m_n)}{|S^{(n)}|(|S^{(n)}|-1)} \sum_{i } \pi_n(i)^2 +  \frac{m_n(m_n-1)}{|S^{(n)}|(|S^{(n)}|-1)}.
\end{split}
\end{equation*}
Summing up, we obtain
\begin{align*}
\E [Y_n] 
=&
\frac{m_n}{|S^{(n)}|} \cdot P( X^{(n)}_0 = X^{(n)}_{ t_{rel}^{(n)}}) 
+
 \frac{m_n(m_n-1)}{|S^{(n)}|(|S^{(n)}|-1)} \cdot P( X^{(n)}_0 \not = X^{(n)}_{ t_{rel}^{(n)}})
\\&-
\left(  \frac{m_n(|S^{(n)}|-m_n)}{|S^{(n)}|(|S^{(n)}|-1)} \sum_{i } \pi(i)^2 +  \frac{m_n(m_n-1)}{|S^{(n)}|(|S^{(n)}|-1)} \right)
\\=&
\frac{m_n(|S^{(n)}|-m_n)}{|S^{(n)}|(|S^{(n)}|-1)} \cdot \left( P( X^{(n)}_0 = X^{(n)}_{ t_{rel}^{(n)}}) -  \sum_{i } \pi(i)^2  \right)
\\ = &
\frac{|S^{(n)}|}{|S^{(n)}|-1} \cdot \Var(f_n) \cdot \left\{ P( X^{(n)}_0 = X^{(n)}_{ t_{rel}^{(n)}}) -  \sum_{i } \pi(i)^2 \right\}.
\end{align*}

Now for each \( n \), fix a function \( f_n^* \) defined as above for which 
\[
\Cov \left( f_n^*(X^{(n)}_0), f_n^*(X^{(n)}_{ t_{rel}^{(n)}}) \right) \leq \E [Y_n].
\]
Then by assumption, \( (f_n^* )_{n \geq 1} \) is noise sensitive with respect to \( (X^{(n)})_{n \geq 1} \).

\end{proof}

\section{The existence of noise stable functions}
In analogue with Proposition~\ref{proposition: noise sensitive functions}, which gives a criteria for when there exist nondegenerate sequences of Boolean  functions  which are  \emph{noise sensitive} with respect to some Markov chain \( (X^{(n)})_{n\geq 1} \), the goal of this section is to obtain criteria for when nondegenerate \emph{noise stable} sequences of Boolean functions exist. In Proposition~\ref{proposition: noise stability existence equivalence}, we  show that the existence of noise stable functions is tightly connected with the so called \emph{delocalization} of eigenvectors which has recently been studied for random graphs and random matrices in eg.~\cite{ab2014}, ~\cite{dll2011},~\cite{dp2012} and~\cite{esy2009}. Proposition~\ref{proposition: vertex transitive functions} then provides a condition which guarantees the existence of nondegenerate noise stable sequence of Boolean functions whenever \( (X^{(n)})_{n \geq 1} \) is a sequence of transitive,  reversible and irreducible   Markov chains. In particular, combining the two propositions, we obtain that eigenvectors of transition  matrices of transitive Markov chains do not localize.

We now state our  main results of this section.

\begin{proposition}\label{proposition: noise stability existence equivalence}
Let \( (X^{(n)})_{n \geq 1} \)  be a sequence of reversible and irreducible  continuous time Markov chains. Then the following two conditions are equivalent.
\begin{enumerate}[label=(\Alph*)]
\item there exists a nondegenerate sequence of Boolean functions which is  noise stable with respect to \( (X^{(n)})_{n \geq 1} \)
%
%
\item there is \( \varepsilon > 0 \), \( k> 0 \) and a sequence \( (\psi^{(n)} )_{ n \geq 1 } \), where \(\psi^{(n)} \in {\Psi_k^{(n)}} \coloneqq \Span( \{ \psi_i^{(n)}\}_{i \colon \lambda_1^{(n)} \leq \lambda_i^{(n)} \leq k \lambda_1^{(n)}}) \) and \( \langle \psi^{(n)}, \psi^{(n)} \rangle = 1 \), such that
\begin{equation}
P_w(\psi^{(n)}(w)^2 < \varepsilon) < 1 - \varepsilon
\label{equation: condition B}
\end{equation}
for all large enough \( n \).
\end{enumerate}
\end{proposition}

Although the previous proposition might seem quite technical, we get the following result as a consequence.

\begin{proposition}
For all sequences of transitive,  reversible and irreducible continuous time Markov chains \( (X^{(n)})_{n \geq 1} \) there is  at least one nondegenerate  sequence of Boolean functions which is noise stable with respect to \( (X^{(n)})_{n \geq 1} \).
\label{proposition: vertex transitive functions}
\end{proposition}

By a \emph{transitive Markov chain} we mean a Markov chain with state space \( S \) and generator \( Q = (q_{ij})_{i,s \in S} \) such that there for any pair of states \( i, j \in S \) is a bijective function \( \varphi \colon S \to S \) such that \( \varphi(i) = j \) and \( q_{\varphi(i) \varphi(j)} = q_{ij} \). Let \( \Aut(X_n) \) denote the set of all such functions \( \varphi \) given some Markov chain \( X_n \). Note that a random walk on a vertex transitive graph is a transitive Markov chain.

We give a proof of Proposition~\ref{proposition: vertex transitive functions} after proving Proposition~\ref{proposition: noise stability existence equivalence}.

\begin{remark}
From Proposition~\ref{proposition: noise stability fourier} it is clear that the existence of noise stable nondegenerate sequences of Boolean functions is equivalent to that there must exists some integer \( k \) and a sequence of vectors \( (\psi^{(n)} )_{n \geq 1 } \) with \({ \psi^{(n)} \in \Psi_{k}^{(n)}} \), such that the distance between  \( \psi^{(n)} \) and \( f_n - \E[f_n] \), for a nondegenerate sequence \( (f_n)_{n \geq 1} \) of Boolean functions, is asymptotically zero. If this holds then clearly \emph{(B)} must happen, so one direction of the proposition above is obvious, although we formalize this argument in Lemma~\ref{lemma: third lemma}. The idea of the proof in the other direction is that if a noise stable sequence of functions exists, then there should be such a sequence which arises from truncating a sequence of functions \( (\psi^{(n)} )_{n\geq 1} \), where \( \psi^{(n)} \in \Psi_k^{(n)} \) for some \( k \geq 1 \). If this is not possible, we cannot have \emph{(B)}. This idea is very similar to the so called  \emph{best threshhold cut} algorithm which sweeps through all possible truncations of an eigenvector corresponding to the first eigenvalue in order to find a good cut in a graph, used in computer science for e.g. clustering (\cite{gm1998}).
\end{remark}

\begin{remark}
If for a sequence \( (X^{(n)} )_{n \geq 1} \) of Markov chains, \( {\lambda_{|S^{(n)}|}^{(n)}/\lambda_1^{(n)} = O(1)}\), then for all large enough \( k \),  \( \Span \big( \psi_0^{(n)}, \Psi_{k}^{(n)} \big) \equiv \mathbb{R}^{|S^{(n)}|} \), which implies that \emph{(B)}  holds in this case. In particular,  Proposition~\ref{proposition: noise stability existence equivalence} is consistent with Remark~\ref{remark: expander graphs}, stating that all sequences of Boolean functions on expander families are noise stable.
\end{remark}

Before moving on to the proof of Proposition~\ref{proposition: noise stability existence equivalence}, it should also be noted that there is in fact sequences of graphs such that no nondegenerate sequence of Boolean functions is noise stable with respect to the corresponding random walk. The next example provides such an example.
\begin{example}
\label{example: small set bottleneck}
Let \( G_n \) be the graph obtained by attaching a complete graph on \( n^2 \) vertices to a complete graph on \( n \) vertices by adding a single edge \( (v,w) \) between a vertex \( c \) in the  larger graph and a vertex \( w \) in the smaller graph. Consider the random walk \(X^{(n)} \) on \( G_n \). Let \( Q^{(n)} = (Q_{ij}^{(n)} )_{i,j \in V(G_n)} \) be the corresponding generator and \( \pi_n \) be the corresponding stationary distribution.
  
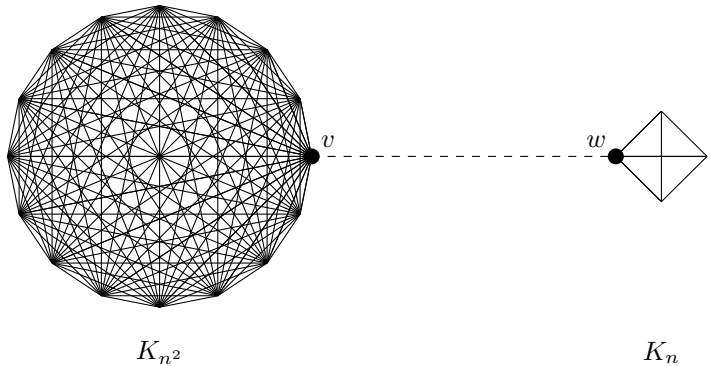
\begin{figure}[H]
 \centering
 \begin{tikzpicture}[scale = 2]

\def\N{15}
\foreach \i in {0,...,\N}{
  \foreach \j in {\i,...,\N}{
    \draw ({cos(\i*360/(\N+1))},{sin(\i*360/(\N+1))}) -- ({cos((\j+1)*360/(\N+1)},{sin((\j+1)*360/(\N+1)});
  };
};
  \fill ({cos(0*360/(\N+1))},{sin(0*360/(\N+1))})  circle (1.5pt);

\draw[dashed] (1,0) -- (3,0);

\def\M{3}
\def\scaling{0.3}
\foreach \i in {0,...,\M}{
  \foreach \j in {\i,...,\M}{
    \draw ({3+\scaling*(1-cos(\i*360/(\M+1)))},{\scaling*sin(\i*360/(\M+1))}) -- ({3+\scaling*(1-cos((\j+1)*360/(\M+1))},{\scaling*sin((\j+1)*360/(\M+1)});
  };
};

  \fill ({3+\scaling*(1-cos(0*360/(\M+1)))},{\scaling*sin(0*360/(\M+1))})  circle (1.5pt);

\draw (1,0) node[above right] {\( v \)};
\draw (3,0) node[above left] {\( w \)};

\draw (-0.0,-1.3) node {\(K_{n^2}\)};
\draw (3.3,-1.3) node {\(K_n\)};

 \end{tikzpicture}
 \end{figure}
 
Write \( K_n \) and \( K_{n^2} \) for the two subgraphs. Then
\begin{equation*}
\begin{cases}
 \vol(G_n) &=  n^2(n^2-1)  + n(n-1) +2 \cr
\pi_n \left(  K_{n} \right) &=({ n \cdot(n - 1)+1})/{\vol(G_n)}  \cr 
\pi_n \left(  K_{n^2} \right) &=({ n^2 \cdot(n^2 - 1)+1})/{\vol(G_n)}  \cr
\pi_n  (v) &= {n^2}/{\vol(G_n)} \cr
\pi_n (w) &=  {n}/{\vol(G_n)} \cr
q^{(n)}_{vw} &= {1}/{n^2 } \cr
q^{(n)}_{wv} &= {1}/{n}
\end{cases}
\end{equation*}
Define a function \( f_n(u) \) for \( u \in V(G_n) \) by setting
\begin{equation*}
f_n(u) = 
\begin{cases}
1 &\textnormal{when } u \in K_{n^2} \cr
\frac{-\pi_n(K_{n^2})}{\pi_n(K_{n})} &\textnormal{when } u \in K_{n}
\end{cases}
\end{equation*}
Then \( \E [f_n(u)] = 0 \), and clearly \( f_n \not \equiv 0 \). Using~\eqref{equation: rayleigh quotient}, we obtain
\begin{equation*}
\begin{split}
\lambda_1^{(n)}
&= 
\min_{f \colon \E[f]=0, f \not \equiv 0} \frac{1}{2} \cdot \frac{\sum_{u_1,u_2 \in V(G_n)} \left| f(u_1) - f(u_2) \right|^2 \pi_n(u_1) \, q_{u_1u_2}^{(n)}}{\langle f,f\rangle}
\\[1ex]
&\leq 
\frac{1}{2} \cdot \frac{\sum_{u_1,u_2\in V(G_n)} \left| f_n(u_1) - f_n(u_2) \right|^2 \pi_n(u_1) \, q_{u_1u_2}^{(n)}}{\langle f_n, f_n \rangle }\\
\\[-2ex]
&=
\frac{1}{2} \cdot \frac{ \left( 1 + \frac{\pi_n(K_{n^2})}{\pi_n(K_{n})} \right)^2 \cdot \left( \pi_n(v) \, q_{vw}^{(n)} + \pi_n(w) \, q_{wv}^{(n)} \right)  }{\pi_n(K_{n^2}) \cdot 1^2 + \pi_n(K_{n}) \cdot \left( \frac{\pi_n(K_{n^2})}{\pi_n(K_{n})}\right)^2} 
\end{split}
\end{equation*}
Using that random walks are reversible, and simplifying, we obtain
\begin{equation*}
\begin{split}
\lambda_1^{(n)} 
&\leq
\frac{1}{2} \cdot \frac{ \left( 1 + \frac{\pi_n(K_{n^2})}{\pi_n(K_{n})} \right)^2 \cdot \left( \pi_n(v) \, q_{vw}^{(n)} + \pi_n(w) \, q_{wv}^{(n)} \right)  }{\pi_n(K_{n^2}) \cdot 1^2 + \pi_n(K_{n}) \cdot \left( \frac{\pi_n(K_{n^2})}{\pi_n(K_{n})}\right)^2} 
\\&= 
 \frac{ \left( 1 + \frac{\pi_n(K_{n^2})}{\pi_n(K_{n})} \right)^2 \cdot \pi_n(v) \, q_{vw}^{(n)}   }{\pi_n(K_{n^2}) + \pi_n(K_{n}) \cdot \left( \frac{\pi_n(K_{n^2})}{\pi_n(K_{n})}\right)^2} 
= 
 \frac{ \left(  \frac{\pi_n(K_n) + \pi_n(K_{n^2})}{\pi_n(K_{n})} \right)^2 \cdot \pi_n(v) \, q_{vw}^{(n)}   }{\pi_n(K_{n^2}) \cdot  \frac{\pi_n(K_{n})  + \pi_n(K_{n^2})}{\pi_n(K_{n})}} 
\\&= 
 \frac{ \left(\pi_n(K_n) + \pi_n(K_{n^2}) \right)  \cdot \pi_n(v) \, q_{vw}^{(n)}   }{\pi_n(K_{n}) \cdot \pi_n(K_{n^2}) } 
= 
 \frac{ 1 \cdot \pi_n(v) \, q_{vw}^{(n)}   }{\pi_n(K_{n}) \cdot \pi_n(K_{n^2}) } 
= 
 \frac{ 1 \cdot n^2/\vol(G_n) \cdot 1/n^2  }{\pi_n(K_{n}) \cdot \pi_n(K_{n^2}) } 
\\&=\frac{1}{\vol(G_n)} \cdot \frac{ 1  }{ \pi_n(K_{n})\cdot \pi_n(K_{n^2}) } \asymp \frac{1}{n^2}.
\end{split}
\end{equation*}

As \( \pi_n (K_{n}) \to 0 \), no nondegenerate sequence of Boolean functions with domains \( (V(G_n))_{n \geq 1} \) can be noise stable with respect to \( (X^{(n)})_{n \geq 1} \) due to its values on \( (K_{n})_{n \geq 1 } \). As the mixing time of \( K_{n^2}\) is of order 1 for each \( n \), while the calculations above show that the the relaxation time of \( G_n \) is of order \(n^2 \), this implies that all nondegenerate sequences of Boolean functions are noise sensitive with respect to \( (X^{(n)})_{n \geq 1} \).
 \end{example}

Not much is known about neither the eigenvectors of transition matrices of Markov chains in general, nor of the eigenvectors of transition matrices  of random walks on graphs in particular. However, in recent years, conditions similar to (B) have been studied for random graphs. To simplify the notation, consider first the following definition, used in~\cite{dp2012} and~\cite{esy2009}. 

\begin{definition}
Let \( T \) be a subset of \( \{ 1,2,\ldots, n \} \) of size \( L \) and let \( \delta > 0 \) be a fixed number.  A vector \( \psi \in \mathbb{R}^n \)  exhibits \emph{\( (T, \delta ) \)-localization} if 
\[
\| \psi(\cdot)  \mathbf{1}_{\cdot \in  T}(\cdot) \|_2^2 \geq (1 - \delta)  \| \psi\|_2^2.
\]
A vector \( \psi \) is said to be \emph{\( (L,\delta)\)-localized} is there is some set \( T \subset \{ 1,2,\ldots, n\} \) with \( \| \mathbf{1}_{T} \|_2^2 \leq L \) such that \( \psi \) is \( (T,\delta) \)-localized.
\end{definition}

Note that the definition above is dependent of the chosen basis. If we let \( \| \cdot \|_2^2 = \langle \cdot, \cdot \rangle \) and note that the opposite of \( (B) \) is given by
\begin{enumerate}
\item[\textit{(\(\neg\)B)}]  \textit{for each \( k> 0 \) and  \( \varepsilon > 0 \)  there is arbitrarily large \( n \) such that }
\begin{equation*}
P(\psi^{(n)}(w)^2 < \varepsilon) \geq 1 - \varepsilon
\end{equation*}
\textit{for all  \(\psi^{(n)} \in \Psi_k^{(n)} \) with \( \langle \psi^{(n)}, \psi^{(n)} \rangle =1\). }
\end{enumerate}
It is now easy to show that \( (\neg B ) \) is equivalent to  
\begin{enumerate}
\item[\textit{(\(\neg\)B)}]  \textit{for each \( k> 0 \) and  \( \varepsilon > 0 \)  there is arbitrarily large \( n \) such that all  \(\psi^{(n)} \in \Psi_k^{(n)} \) is \( (\varepsilon n, \varepsilon)\)-localized.}
\end{enumerate}

For the usual \( L_2 \)-norm on \( \mathbb{R}^n \), which agrees with \( \langle \cdot , \cdot\rangle \) only for Markov chains with uniform stationary distribution \( \pi \), the following is known.

\paragraph{\textbf{Corollary 3.2 in~\cite{dll2011}}} \textit{For every \( p > 0 \) there is \( \delta = \delta(p)>0 \) and \( \varepsilon(p)\in(0,0.5) \) such that for almost all \( G \sim G(n,p) \), no eigenvector of the adjency matrix except \( \psi_0 \) is \( (\varepsilon n,\delta)\)-localized. }

\paragraph{\textbf{Theorem 3 in~\cite{dp2012}}}%
\textit{Fix \( \delta > 0 \). Let \( d_n = (\log n)^\gamma \) for some \( \gamma > 0 \) and let \( T_n \) be a deterministic sequence of sets of size }
\[
L_n = o(2/(\exp(d_n^{-\alpha}) - \exp(-d_n^{-\alpha})) )
\]
\textit{where  \( \alpha \in (0,\min (1,1/\gamma) \). Then if \( A_n \) is the adjency matrix of a random regular graph with valency \( d_n \), then}
\[
P(\textnormal{no eigenvector of } A_n \textnormal{ is } (T_n,\delta) \textnormal{-localized}) \geq 1-o(1/d_n).
\]

\paragraph{\textbf{Theorem 7.1 in~\cite{esy2009}}}%
\textit{Let \( H \) be an \( n \times  n \) Hermitean random matrix from a Wigner ensamble satisfying two technical conditions (see~\cite{esy2009}). Suppose that \( L \in [n]\) and \( \eta > 0 \)  are chosen such that \( \eta \) and \( \nu \coloneqq L/n \) is sufficiently small. Then there is a constant \( c > 0 \) which do not depend on \( n \) such that }
\[
P(\textnormal{no eigenvector of \( H \) is } (L,\eta) \textnormal{-localized}) \geq 1-e^{-cn}.
\]

Note that none of the results above is exactly what we need. The first and third results are both of the type we need. However, in the first case, the result is for eigenvectors of the adjency matrix instead of for the generator and in addition, the norm is wrong. In the third case, the result is not for transition matrices of Markov chains at all, but is of the correct type. In the second case,  the eigenvectors are eigenvectors for the generator as well as the random graphs are regular, and for the same reason the norm is correct as well. However, there are two other problems with this result. Firstly, we need a result for all sequences of sets and not for a beforehand chosen sequence. More problematic for us  however, is that the measure of the sets \( T_n \) will tend to zero. 
 
Given what is currently known however, there seems to be nothing preventing that \( (B)\) would hold for a large family of Markov chains, and further results in this direction would be interesting in the light of Proposition~\ref{proposition: noise stability existence equivalence}.

We will now give a proof of Proposition~\ref{proposition: noise stability existence equivalence} through a sequence of lemmas. The first of these lemmas is particularly interesting since it provides means by which one can validate that a sequence of truncated real-valued functions is noise stable.

\begin{lemma}
Let \( (X^{(n)})_{n \geq 1 } \) be a sequence of   reversible and irreducible continuous time Markov chains, and let \( g(\varepsilon) \) be a real valued function such that \( \varepsilon /g(\varepsilon)^2 \to 0 \) as \( \varepsilon \to 0 \).
Suppose that there is \( k > 0 \), \( \delta>0 \), a sequence \( (\psi^{(n)})_{n \geq 1 } \) where \( \psi^{(n)} \in \Psi_k^{(n)} \) and \( \langle \psi^{(n)}, \psi^{(n)} \rangle = 1 \), and a sequence \( ( c_n)_{n \geq 1} \) such that for all \( n \geq 1 \),
\begin{enumerate}[label=(\roman*)]
\item \( P(\psi^{(n)}(w) \geq c_n) > \delta \)
\item \( P(\psi^{(n)}(w) < c_n) > \delta \)
\item \( \lim_{\varepsilon \to 0} \sup_n P(\psi^{(n)}(w) \in [c_n -{g(\varepsilon)}, c_n + {g(\varepsilon)}] ) = 0 \)
\end{enumerate}
Then the sequence  \( ( \mathbf{1}_{\psi^{(n)} \geq c_n} )_{n\geq 1}\) is nondegenerate and noise stable with respect to \( (X^{(n)})_{n\geq 1} \).
\label{lemma: technical part of first lemma}
\end{lemma}

\begin{proof}
Note first that for all functions \( f \) with range \( \{ 0,1 \} \),
\begin{equation}
P(f(X_0) \not = f(X_{t})) = \E \left[ (f(X_0)-f(X_{t}))^2\right] = 2\sum_{i=1}^{|S^{(n)}|-1} (1-e^{-  \lambda_i t}) \hat f(i)^2
\label{equation: inf norm and stability 1}
\end{equation}
where the last equality, by the proof of Proposition~\ref{proposition: noise sensitivity fourier} in fact holds for all real valued functions with the same domain as \( f \).

Set \( J_n(\varepsilon) =  [c_n - g(\varepsilon), c_n + g(\varepsilon)] \) and \( f_n = \mathbf{1}_{\psi^{(n)} \geq c_n} \). Then \( ( f_n)_{n \geq 1 } \) is nondegenerate by (i) and (ii).

To show that \( (f_n)_{n \geq 1} \) will be noise stable with respect to \( (X^{(n)})_{n \geq 1} \), note that
\begin{equation*}
\begin{split}
P(f_n(X_0^{(n)}) \not = f_n(X^{(n)}_{\varepsilon t_{rel}^{(n)}})) 
&= 
\E \left[ (f_n(X^{(n)}_0)-f_n(X^{(n)}_{\varepsilon t_{rel}^{(n)}}))^2\right] 
\\&\leq 
P(\psi^{(n)}(X^{(n)}_0) \in J_n(\varepsilon) ) +P\left(\left| \psi^{(n)}(X_0^{(n)})-\psi^{(n)}(X^{(n)}_{\varepsilon t_{rel}^{(n)}}) \right|  \geq {g(\varepsilon)} \mid \psi^{(n)}(X_0^{(n)}) \not \in J_n(\varepsilon)  \right)
\\&\leq
P(\psi^{(n)}(X^{(n)}_0) \in J_n(\varepsilon) ) + \frac{\E \left[ (\psi^{(n)}(X_0^{(n)})-\psi^{(n)}(X^{(n)}_{\varepsilon t_{rel}^{(n)}}))^2  \mid \psi^{(n)}(X_0^{(n)}) \not \in J_n(\varepsilon)  \right]}{g(\varepsilon)^2}
\\&\leq
P(\psi^{(n)}(X^{(n)}_0) \in J_n(\varepsilon) ) + \frac{\E \left[ (\psi^{(n)}(X_0^{(n)})-\psi^{(n)}(X^{(n)}_{\varepsilon t_{rel}^{(n)}}))^2   \right]}{g(\varepsilon)^2 P( \psi^{(n)}(X_0^{(n)}) \not \in J_n(\varepsilon) )}
\\&\leq
P(\psi^{(n)}(X^{(n)}_0) \in J_n(\varepsilon) ) + \frac{2\sum_i (1-e^{-  \varepsilon \lambda_i t_{rel}^{(n)}}) \hat \psi^{(n)}(i)^2}{g(\varepsilon)^2 P( \psi^{(n)}(X_0^{(n)}) \not \in J_n(\varepsilon) )}
\\&\leq
P(\psi^{(n)}(X^{(n)}_0) \in J_n(\varepsilon) ) + \frac{2(1-e^{- \varepsilon k})  \sum_i  \hat \psi^{(n)}(i)^2}{g(\varepsilon)^2 P( \psi^{(n)}(X_0^{(n)}) \not \in J_n(\varepsilon) )}
\\&=
P(\psi^{(n)}(X^{(n)}_0) \in J_n(\varepsilon) ) + \frac{2(1-e^{- \varepsilon k}) }{g(\varepsilon)^2 P( \psi^{(n)}(X_0^{(n)}) \not \in J_n(\varepsilon) )}
\\&\leq
P(\psi^{(n)}(X^{(n)}_0) \in J_n(\varepsilon) ) + \frac{2\varepsilon k}{g(\varepsilon)^2 P( \psi^{(n)}(X_0^{(n)}) \not \in J_n(\varepsilon) )}.
\end{split}
\end{equation*}
Using (iii), we obtain
\begin{equation*}
\lim_{\varepsilon \to 0} \sup_n P(f_n(X_0^{(n)}) \not = f_n(X^{(n)}_{\varepsilon t_{rel}^{(n)}})) 
=
\lim_{\varepsilon \to 0} \sup_n
\left\{ P(\psi^{(n)}(w) \in J_n(\varepsilon) ) + \frac{2\varepsilon k}{g(\varepsilon)^2 P( \psi^{(n)}(X_0^{(n)}) \not \in J_n(\varepsilon) )} \right\}
=
0
\end{equation*}
i.e. \( (f_n)_{n \geq 1 } \) is noise stable with respect to \( (X^{(n)})_{n \geq 1} \).
\end{proof}

\begin{lemma}
\label{lemma: first lemma}
Let \( (X^{(n)})_{n \geq 1} \) be a sequence of reversible and irreducible continuous time Markov chains. Then either
\begin{enumerate}
\item[(C)]  there is  \( k\geq 1 \), a sequence  \( (c_n)_{n \geq 1} \) and a sequence \( (\psi^{(n)})_{n \geq 1} \), where \( \psi^{(n)} \in \Psi_k^{(n)} \) such that \( ( \mathbf{1}_{\psi^{(n)}(w)>c_n} )_{n \geq 1} \) is nondegenerate and noise stable with respect to \( (X^{(n)})_{n \geq 1} \), or
\item[(\(\neg\)B)]  for each \( k> 0 \) and  \( \varepsilon > 0 \)  there is arbitrarily large \( n \) such that  
\begin{equation*}
P(\psi^{(n)}(w)^2 < \varepsilon) \geq 1 - \varepsilon
\end{equation*}
for all  \(\psi^{(n)} \in \Psi_k^{(n)} \) with \( \langle \psi^{(n)}, \psi^{(n)} \rangle =1\). 
\end{enumerate}
Note in particular that this entails that \( (B) \Rightarrow (C) \Rightarrow (A) \).
\end{lemma}

\begin{proof}
For each \( \delta > 0 \), \( n \geq 1\), \( k > 0 \) and \( \psi^{(n)} \in \Psi_k^{(n)} \), define
\[
c_\delta(\psi^{(n)}) = \sup \{ c \in \mathbb{R} \colon P(\psi^{({n})}(w) \leq c ) < \delta \} 
\]
and
\[
c_{1-\delta}(\psi^{(n)})  = \inf \{ c \in \mathbb{R} \colon P(\psi^{({n})}(w) \geq c )  < \delta \}  .
\]
and set \(  I_\delta(\psi^{(n)}) \coloneqq [ c_\delta^{(n)}(\psi) , c_{1-\delta}(\psi^{(n)})] \). Note in particular that this implies that
\begin{equation}\label{equation: measure of I}
P(\psi^{({n})}(w) \in I_\delta(\psi^{(n)})) \geq 1 - 2\delta.
\end{equation}

Suppose that \( (C) \) does not hold. Then the assumptions of Lemma~\ref{lemma: technical part of first lemma} cannot be satisfied, and in particular, for all \( k > 0 \) and \( \delta > 0 \), there is \( \alpha = \alpha (k, \delta) > 0 \) such that 
\[
\limsup_{\varepsilon' \to 0}\sup_n \inf_{c \in I_\delta (\psi^{(n)})}P(\psi^{(n)}(w) \in [c -\sqrt{\varepsilon'}, c + \sqrt{\varepsilon'}] ) \geq \alpha
\]
for all \( (\psi^{(n)})_{n\geq 1} \) with \( \psi^{(n)} \in \Psi_k^{(n)} \) and \( \langle \psi^{(n)}, \psi^{(n)} \rangle = 1\). 

By Lemma~\ref{lemma: technical lemma on the combination of limits and sups}, this implies that for any \( \varepsilon' > 0 \)  there is  \( n(\varepsilon',\delta)> 0 \) such that
\[
 \inf_{c \in I_\delta (\psi^{(n(\varepsilon',\delta))})}P(\psi^{(n(\varepsilon',\delta))}(w) \in [c -\sqrt{\varepsilon'}, c + \sqrt{\varepsilon'}] ) \geq \alpha/2
\]
for all \( \psi^{(n(\varepsilon',\delta))} \in \Psi_k^{(n(\varepsilon',\delta))} \). This in turn implies that for any \( \psi^{(n(\varepsilon',\delta))} \in \Psi_k^{(n(\varepsilon',\delta))} \) and for all \( c \in I_\delta (\psi^{(n(\varepsilon',\delta))}) \)
\begin{equation}\label{eq: weight of interval}
P(\psi^{(n(\varepsilon',\delta))}(w) \in [c -\sqrt{\varepsilon'}, c + \sqrt{\varepsilon'}] ) \geq \alpha/2.
\end{equation}
As the length of \(  I_{\delta}(\psi^{(n(\varepsilon',\delta))}) \) is \( c_{1-\delta}^{(n(\varepsilon',\delta))} - c_{\delta}^{(n(\varepsilon',\delta))} \) and any interval of length \( 2 \sqrt{\varepsilon'} \) with midpoint contained in \(  I_{\delta}(\psi^{(n(\varepsilon',\delta))}) \) has measure at least \( \alpha / 2 \) by~\ref{eq: weight of interval}, we obtain
\begin{equation*}
1 \geq  P(\psi^{(n(\varepsilon',\delta))}(w) \in I_{\delta}(\psi^{(n(\varepsilon',\delta))})) 
\geq
\frac{c_{1-\delta}^{(n(\varepsilon',\delta))} - c_{\delta}^{(n(\varepsilon',\delta))}}{2\sqrt{\varepsilon'}} \cdot  \frac{\alpha}{2}
=
\frac{|I_\delta(\psi^{(n(\varepsilon',\delta))})|}{\sqrt{\varepsilon'}} \cdot  \frac{\alpha}{4}
\end{equation*}
or equivalently, that
\begin{equation}
|I_\delta(\psi^{(n(\varepsilon',\delta))}) | \leq\frac{4}{\alpha} \cdot \sqrt{\varepsilon'}.
\label{equation: interval width}
\end{equation}

Now recall that we need to show that \( (\neg B) \) holds, i.e. that there for any \( k>0 \) and \( \varepsilon > 0 \) is arbitrarily large \( n \) such that \( P(\psi^{(n)}(w)^2<\varepsilon) \geq 1 - \varepsilon \) for all \( \psi^{(n)} \in \Psi_k^{(n)} \) with \( \langle \psi^{(n)}, \psi^{(n)} \rangle = 1 \).
To this end, fix \( k>0 \) and \(\varepsilon >0\) and pick a sequence \( (\delta_j)_{j \geq 1} \) such that \( 2 \delta_j/(1-2\delta_j)^2 < \varepsilon/2 \) for all \( j \geq 1 \), and \( \lim_{j \to \infty} \delta_j = 0 \). Then by~\eqref{equation: measure of I},
\begin{equation}\label{equation: measure of I, 2}
P(\psi^{({n})}(w) \in I_{\delta_j}(\psi^{(n)})) \geq 1-2\delta_j > 1 - \varepsilon/2 > 1 - \varepsilon
\end{equation}
for all \( n \geq 1 \) and all sequences \( (\psi^{(n)})_{n \geq 1} \) with \( \psi^{(n)} \in \Psi_k^{(n)} \).

Pick \( \varepsilon'(j) \) such that \( \frac{4}{\alpha(k,\delta_j)} \cdot \sqrt{\varepsilon'(j)} \leq \varepsilon/2 \). By the derivations above, we can find \( n(\varepsilon'(\delta_j), \delta_j) \) such that 
\begin{equation}
|I_{\delta_j}(\psi^{(n(\varepsilon'(\delta_j),\delta_j))}) | \leq\frac{4}{\alpha(k,\delta_j)} \cdot \sqrt{\varepsilon'(\delta_j)} = \varepsilon/2
\label{equation: interval width conclusion}
\end{equation}
for all \( \psi^{(n(\varepsilon'(\delta_j),\delta_j))} \in \Psi^{(n(\varepsilon'(\delta_j),\delta_j))}_k \).

By applying Lemma~\ref{lemma: delocalization around zero} to the sets 
\[
 A_j = \{ w \in S^{((n(\varepsilon'(\delta_j),\delta_j)))} \colon \psi^{((n(\varepsilon'(\delta_j),\delta_j)))}(w) \in I_{\delta_j}(\psi^{((n(\varepsilon'(\delta_j),\delta_j)))}) \} 
 \]
 we get
\[
 \E[\psi^{(n(\varepsilon'(\delta_j),\delta_j))}(w) \mid \psi^{(n(\varepsilon'(\delta_j),\delta_j))}(w) \in I_{n(\varepsilon'(\delta_j),\delta_j)}]^2 \leq \frac{2\delta_j}{(1-2\delta_j)^2} \leq \varepsilon/2.
\]
Combining this with~\eqref{equation: interval width conclusion}, we obtain
\[
I_{\delta_j}(\psi^{(n(\varepsilon'(\delta_j),\delta_j))}) \subseteq [-\varepsilon/2 - \varepsilon/2 , \varepsilon/2+\varepsilon/2] = [-\varepsilon, \varepsilon].
\]
Summing up, we  have that
\begin{equation*}
\begin{split}
P(\psi^{({n(\varepsilon'(\delta_j),\delta_j)})}(w)^2<\varepsilon) &\geq P(|\psi^{({n(\varepsilon'(\delta_j),\delta_j)})}(w)|<\varepsilon) 
\\&\geq
P(|\psi^{(n)}(w)|<\varepsilon \mid \psi^{(n)}(w) \in I_\delta(\psi^{(n)})) \cdot P(\psi^{(n)}(w) \in I_\delta(\psi^{(n)}))
\\&\geq
1 \cdot (1-2\delta)
\\&\geq 1 - \varepsilon
\end{split}
\end{equation*}
for all sequences \( (\psi^{(n(\varepsilon'(\delta_j),\delta_j))})_{j \geq 1} \) with \( \psi^{(n(\varepsilon'(\delta_j),\delta_j))} \in \Psi_k^{(n(\varepsilon'(\delta_j),\delta_j))} \). 

The only thing which remains to do to finish the proof is to show that \( n(\varepsilon'(\delta_j),\delta_j) \) can be chosen arbitrarily large. To do this, it is enough to show that
\[
\lim_{j \to \infty}n(\varepsilon'(\delta_j),\delta_j) = \infty.
\]
To see this, note that if \( \delta < \min_{w \in S^{(n)}} \pi_n(w) \), then \( \psi^{(n)} \in I_\delta \) for all \( w \in S^{(n)} \), which implies that \( |I_\delta| > 0 \) as \( \langle \psi^{(n)}, \psi^{(n)} \rangle = 1 \) and \( \E [ \psi^{(n)} ]= 0 \). This contradicts that \( |I_\delta| < \frac{4}{\alpha} \cdot \sqrt{\varepsilon'} < \frac{\varepsilon}{2} \). As \( \lim_{j \to \infty } \delta_j  = 0 \) by assumption, this shows that
\[
n(\varepsilon'(\delta_j),\delta_j) \geq \inf \{ n \geq 1 \colon \delta \geq  \min_{w \in S^{(n)}} \pi_n(w) \}.
\]

This completes the proof.

\end{proof}

\begin{lemma}
\label{lemma: third lemma}
Let \( (X^{(n)})_{n \geq 1} \) be a sequence of reversible and irreducible continuous time Markov chains. Then if \( (f_n)_{n \geq 1} \) is a sequence of nondegenerate Boolean functions, both \( (A) \) and \( ( \neg B) \) cannot hold. In particular \( (A) \Rightarrow (B) \).
\end{lemma}

\begin{proof}
Suppose that \( (f_n)_{n \geq 1} \) is a nondegenerate sequence of Boolean functions which is noise stable with respect to \( (X^{(n)})_{n \geq 1} \). 
As \( (f_n)_{n \geq 1 } \) is nondegenerate and Boolean, there is \( \delta > 0 \) such that
\[
\delta < \Var (f_n) = \E[f_n^2]- \E [f_n]^2 = \E[f_n] - \E[f_n]^2 = \E[f_n] \cdot ( 1 - \E[f_n] ) 
\]
for all \( n \geq 1 \). Rewriting this, we obtain
\[
\left| \E[f_n] - \frac{1}{2}\right| < \sqrt{\frac{1}{4} - \delta}
\]
which in particular implies that
\begin{equation}
\min\left\{ \left| \E[f_n] -0\right|, \left| \E[f_n] - 1\right| \right\} > \frac{1}{2} - \sqrt{\frac{1}{4}-\delta}
\label{equation: lemma 3 distance to Boolean}
\end{equation}
for all \( n \geq 1 \).

As \( (f_n)_{n \geq 1 } \) is noise stable with respect to \( (X^{(n)})_{n \geq 1} \), by Proposition~\ref{proposition: noise stability fourier},  for all \( \alpha >0 \) there is \( k > 0 \) such that
\begin{equation}
\sum_{i \colon \lambda_i > k \lambda_1} \hat f_n(i)^2 < \alpha
\label{equation: lemma 3 noise stability}
\end{equation}
for all \( n \geq 1 \). Fix such a \( k \). 
For a function \( f \colon S^{(n)} \to \mathbb{R} \), let \( P_kf \coloneqq  \sum_{i \geq 1\colon \lambda_i < k \lambda_1} \hat f(i) \psi_i^{(n)} \) and let  \( P_{k,0}f \coloneqq  \sum_{i \geq 0\colon \lambda_i < k \lambda_1} \hat f(i) \psi_i^{(n)} \).
If \( (\neg B) \) holds, then for all \( \varepsilon > 0 \) there is arbitrarily large \( n \geq 1 \) such that for all \( \psi^{(n)} \in \Psi_k^{(n)} \) with \( \langle \psi^{(n)}, \psi^{(n)} \rangle = 1 \),
\[
P(\psi^{(n)} (w)^2<\varepsilon) \geq 1 - \varepsilon.
\]
In particular
\[
P(P_kf_n(w)^2 < \varepsilon ) \geq P(P_kf_n(w)^2 < \varepsilon \langle P_kf_n, P_kf_n \rangle) \geq 1-\varepsilon.
\]
Using~\eqref{equation: lemma 3 distance to Boolean} we now obtain
\begin{equation*}
\begin{split}
\sum_{i \colon \lambda_i > k \lambda_1} \hat f_n(i)^2 
&= 
\langle f_n-P_{k,0}f_n, f_n-P_{k,0}f_n \rangle 
\\&= 
\E[( f_n-P_{k,0}f_n )^2] 
\\&\geq 
P(P_kf_n(w)^2<\varepsilon) \cdot 
\E[( f_n-P_{k,0}f_n )^2 \mid P_kf_n(w)^2<\varepsilon]
\\&\geq
(1-\varepsilon) \cdot \left( \frac{1}{2} - \sqrt{\frac{1}{4}-\delta} - \sqrt{\varepsilon} \right).
\end{split}
\end{equation*}
If \( \alpha \) and \( \varepsilon \) are both chosen small enough, this contradicts~\eqref{equation: lemma 3 noise stability}. This finishes the proof.
\end{proof}

\begin{proof}[Proof of Proposition~\ref{proposition: noise stability existence equivalence}]
Note first that it is trivial that \( (C) \Rightarrow (A) \). By Lemma~\ref{lemma: third lemma}, \( (A) \) implies \( (B) \).  By Lemma~\ref{lemma: first lemma}, \( \neg (C) \Rightarrow \neg (B) \), in turn implying that \( (B) \Rightarrow (C)   \). Combining these results, we obtain \( (A) \Rightarrow (B) \Rightarrow C \Rightarrow A\), implying that \( (A) \Leftrightarrow (B)\) which is the desired conclusion.
\end{proof}

We now proceed to the proof of Proposition~\ref{proposition: vertex transitive functions}.

\begin{proof}[Proof of Proposition~\ref{proposition: vertex transitive functions}]
We will prove something stronger than what is required to be able to conclude the proposition, namely that there is \( \varepsilon > 0 \) and  a sequence \( (\psi^{(n)} )_{n \geq 1} \), where \( \psi^{(n)} \in \Psi_1^{(n)} \) and \( \langle \psi^{(n)}, \psi^{(n)} \rangle = 1 \) for each \( n \), such that 
\[
P_w( \psi^{(n)}(w)^2 < \varepsilon ) < 1 - \varepsilon
\]
for all large enough \( n \).

Assume for contradiction that this is not true, i.e. assume that for any \( \varepsilon > 0 \) there is a subsequence \( (n_k^{(\varepsilon)})_{k \geq 1} \), such that
\begin{equation}
P_w(\psi^{(n_k^{(\varepsilon)})}(w)^2 < \varepsilon) \geq 1 - \varepsilon
\label{equation: contradiction assumption}
\end{equation}
for all  sequences  \( ( \psi^{(n_k^{(\varepsilon)})})_{k \geq 1} \) with \( \psi^{(n_k^{(\varepsilon)})} \in \Psi_1^{(n_k^{(\varepsilon)})} \). We can without loss of generality assume that
\[
\{ n_k^{(\varepsilon)}\}_{k\geq 1}  \subseteq \{ n_k^{(\varepsilon')}\}_{k\geq 1} 
\]
whenever \( \varepsilon < \varepsilon' \). To make the notation less cumbersome, we will write \( n_k \) instead of \( n_k^{(\varepsilon)} \) whenever the dependency on \( \varepsilon \) is clear.

Set \( m^{(n_k)} \coloneqq \dim \Psi_1^{(n_k)} \) and let \( \chi_1^{(n_k)} \), \ldots, \( \chi_{m^{(n_k)}}^{(n_k)} \) be an orthonormal basis for \( \Psi_1^{(n_k)} \).
Define
\[
\phi_1^{(n_k)}[\theta_1] = \chi^{(n_k)}_1 \sin \theta_1 + \chi^{(n_k)}_2 \cos \theta_1
\]
and inductively for \( j = 2, 3, \ldots, m^{(n_k)} \),
\begin{equation*}
\phi_j^{(n_k)}[\theta_1, \theta_2, \ldots, \theta_{j}]
= 
\phi_{j-1}^{(n_k)}[\theta_1,\ldots, \theta_{j}]\sin \theta_{j} + \chi_{j+1}^{(n_k)} \cos \theta_{j}.
\end{equation*}
Then all vectors  \( \psi \in  \Psi_1^{(n_k)} \) of length one can be written as 
\[
 \psi  [\theta_1, \ldots, \theta_{m^{(n_k)}-1}]  \coloneqq \phi_{m^{(n_k)}}^{(n_k)} [\theta_1, \ldots, \theta_{m^{(n_k)}-1}] 
 \]
 for some tuple \( (\theta_1, \ldots, \theta_{m^{(n_k)}-1}) \in [0,\pi] \times [0,\pi/2]^{m^{(n_k)}-2} \). 
Note that \( \psi \) is the parametric description of a \( m^{(n_k)} \)-dimensional sphere in \( \mathbb{R}^{|S^{(n_k)}|} \). Let \( \mu^{(n_k)} \) be its surface measure.

Define
\[
\Theta_w^{(n_k^{(\varepsilon)})}(\varepsilon) = \left\{ (\theta_1, \ldots, \theta_{m^{(n_k)}-1})\in [0,\pi] \times [0,\pi/2]^{m^{(n_k)}-2} \colon \phi_{m^{(n_k)}}^{(n_k)}[\theta_1, \ldots \theta_{m^{(n_k)}-1}](w)^2< \varepsilon  \right\}
\]
\eqref{equation: contradiction assumption} is then equivalent to that
\[
P(\{ w \colon (\theta_1, \ldots, \theta_{m^{(n_k)}-1} ) \in \Theta_w^{(n_k^{(\varepsilon)})} \}) \geq 1 - \varepsilon
\]
for each tuple \( (\theta_1, \ldots, \theta_{m^{(n_k)}-1} ) \).
By~\eqref{equation: contradiction assumption}, each \( \psi \in \Psi_1^{(n_k)} \cap S_{|S^{(n_k)}|-1}\) corresponds to a tuple \( (\theta_1, \ldots, \theta_{m^{(n_k)}-1} )\) which lies in \( \Theta_w^{(n_{k})}(\varepsilon) \) for at least \( (1-\varepsilon)|S^{(n_k)}| \) different \( w\in S^{(n_k)} \). 
In terms of probabilities, this means that 
\[
P(\{ w \colon (\theta_1, \ldots, \theta_{m^{(n_k)}-1} ) \in \Theta_w^{(n_k^{(\varepsilon)})} \}) \geq 1 - \varepsilon
\]
for each tuple \( (\theta_1, \ldots, \theta_{m^{(n_k)}-1} ) \). As this holds for all such tuples, we obtain that
\[
P \times \mu^{(n_{k})} (\{ w \colon (\theta_1, \ldots, \theta_{m^{(n_k)}-1} ) \in \Theta_w^{(n_k^{(\varepsilon)})} \}) \geq 1 - \varepsilon
\]
As \( P \) a finite probability measure with is uniform on \(S^{(n_m)} \), there must be at least one \( w \in S^{(n_k)} \) such that \( \mu^{(n_{k})} (\Theta_w^{(n_{k})}(\varepsilon)) \geq 1-\varepsilon  \). 
Fix such an element  \(w \in S^{(n_k)} \). 

As \( \mu^{(n_k)} \) is a surface measure, it is invariant of the choice of basis. In particular, we can assume that the basis is chosen in such a way that 
\[
 \chi^{(n_k)}_1(w) = \ldots =\chi^{(n_k)}_{m^{(nk)}-1}(w) = 0.
\]
Then \( \mu( \Theta_w^{(n_{k})}(\varepsilon) )  \) 
can be rewritten as
\begin{equation*}
\begin{split}
\mu^{(n_k)} (\Theta_w^{(n_{k})}(\varepsilon)) 
&= 
\mu^{(n_k)} \left( \left\{ (\theta_1, \ldots, \theta_{m^{(n_k)}-1}) \colon \phi_{m^{(n_k)}}^{(n_k)} [\theta_1, \ldots, \theta_{m^{(n_k)}-1}](i)^2< \varepsilon \right\}\right)
\\&=
\mu^{(n_k)} \left( \left\{ (\theta_1, \ldots, \theta_{m^{(n_k)}-1}) \colon \chi^{(n_k)}_{m^{(n_k)}}(w)^2\cos^2(\theta_{m^{(n_k)}-1})<\varepsilon  \right\}\right)
\\&=
\mu^{(n_k)} \left( \left\{ (\theta_1, \ldots, \theta_{m^{(n_k)}-1}) \colon \cos^2(\theta_{m^{(n_k)}-1})< \frac{\varepsilon}{\chi^{(n_k)}_{m^{(n_k)}}(w)^2}  \right\}\right)
\\&=
\mu^{(n_k)} \left( \left\{ (\theta_1, \ldots, \theta_{m^{(n_k)}-1}) \colon \cos^2(\theta_{m^{(n_k)}-1})< \frac{\varepsilon}{M_w^{(n_k)}}  \right\}\right)
\end{split}
\end{equation*}
where \( M_w^{(n_k)} = \sup_{\theta_1,\ldots, \theta_{m^{(n_k)} -1}} \phi_{m^{(n_k)}}^{(n_k)} [\theta_1, \ldots, \theta_{m^{(n_k)}-1}]^2 \).


Since the surface element of a \( m^{(n_k)} \)-dimensional sphere is given by 
\[
\sin(\theta_2) \sin^2(\theta_3) \cdots \sin^{m^{(n_k)}-2}(\theta_{m^{(n_k)}-1}) d\theta_1 d\theta_2\ldots d\theta_{m^{(n_k)}-1}
\]
we have that
\begin{equation*}
\mu(\Theta_w^{(n_{k})}(\varepsilon))
= \frac{\int_{\arccos \sqrt{\varepsilon/M_w^{(n_k)}}}^{\pi/2} \sin^{m^{(n_k)}-2}(\theta_{m^{(n_k)}-1}) d\theta_{m^{(n_k)}-1}}{\int_{0}^{\pi/2} \sin^{m^{(n_k)}-2}(\theta_{m^{(n_k)}-1}) d\theta_{m^{(n_k)}-1}}.
\end{equation*}
To be able to give bounds on this integral, note first that
\[
 \arccos \sqrt\frac{\varepsilon}{M_w^{(n_k)}} = \frac{\pi}{2} - \sqrt\frac{\varepsilon}{M_w^{(n_k)}} + O\left(\left(\sqrt\frac{\varepsilon}{M_w^{(n_k)}}\right)^3\right).
\]
This implies that
\begin{equation*}
\begin{split}
\mu^{(n_k)}(\Theta_w^{(n_{k'})}(\varepsilon)) 
=
 \frac{\int_{0}^{\sqrt{\varepsilon/M_w^{(n_k)}}+ O((\varepsilon/M_w^{(n_k)})^{3/2})} \cos^{m^{(n_k)}-2}(\theta) \,d\theta}{\int_{0}^{\pi/2} \cos^{m^{(n_k)}-2}(\theta) \,d\theta}.
\end{split}
\end{equation*}

For any \(a \in (0,1) \),
\begin{equation*}
\begin{split}
 \cos^{m^{(n_k)}-2}\theta = a 
\Leftrightarrow \theta &= \arccos a^{1/(m^{(n_k)}-2)}
\\&=
\sqrt{2(1-a^{1/(m^{(n_k)}-2)})} + O((1-a^{1/(m-2)})^{3/2})
\\&=
 \sqrt\frac{-2 \log a}{m^{(n_k)}-2}+  O(\frac{\log a}{m^{(n_k)}-2}) +  O((1-a^{1/(m-2)})^{3/2}) 
\\&=
 \sqrt\frac{-2 \log a}{m^{(n_k)}-2}+  O(\frac{\log a}{m^{(n_k)}-2}) .
\end{split}
\end{equation*}
Using Lemma~\ref{lemma: upper bound on measure}, we obtain
\begin{equation*}
\begin{split}
\frac{\int_{0}^{\frac{1}{2} \cdot \sqrt\frac{-2 \log a}{m^{(n_k)}-2}+  O(\frac{\log a}{m^{(n_k)}-2}) } \cos^{m^{(n_k)}-2}(\theta) \,d\theta}{\int_{0}^{\pi/2} \cos^{m^{(n_k)}-2}(\theta) \,d\theta}
\leq
\frac{1}{1+a}.
\end{split}
\end{equation*}
Since \( (X^{(n)} )_{n \geq 1} \) is transitive, Lemma~\ref{lemma: maximal value} implies that \( M_w^{(n_k)} = \dim \Psi^{(n_k)} = m^{(n_k)} \). In particular, this implies that if we choose \( a = \exp(-4\varepsilon)  \), then
\begin{equation*}
\begin{split}
\mu^{(n_k)}(\Theta_w^{(n_{k'})}(\varepsilon)) 
&=
 \frac{\int_{0}^{\sqrt{\varepsilon/M^{(n_k)}}+ O((\varepsilon/M^{(n_k)})^{3/2})} \cos^{m^{(n_k)}-2}(\theta) \,d\theta}{\int_{0}^{\pi/2} \cos^{m^{(n_k)}-2}(\theta) \,d\theta}
\\&= 
 \frac{\int_{0}^{\sqrt{\varepsilon/m^{(n_k)}}+ O((\varepsilon/m^{(n_k)})^{3/2})} \cos^{m^{(n_k)}-2}(\theta) \,d\theta}{\int_{0}^{\pi/2} \cos^{m^{(n_k)}-2}(\theta) \,d\theta}
\\&\leq 
\frac{\int_{0}^{\frac{1}{2} \cdot \sqrt\frac{-2 \log a}{m^{(n_k)}-2}+  O(\frac{\log a}{m^{(n_k)}-2}) } \cos^{m^{(n_k)}-2}(\theta) \,d\theta}{\int_{0}^{\pi/2} \cos^{m^{(n_k)}-2}(\theta) \,d\theta}
\\&\leq
\frac{1}{1+\exp(-4\varepsilon)}
\end{split}
\end{equation*}
whenever \( \varepsilon\) is small enough. In particular, if \( \varepsilon \) is small enough, then we cannot have that 
\[
\mu(\Theta_i^{(n_{k'})} (\varepsilon)) \geq 1-\varepsilon
\]
for any \( k \geq 1 \). This completes the proof.

\end{proof}

\section{Noise stability and the bottleneck ratio}

Although Proposition~\ref{proposition: vertex transitive functions} tell us that for a relatively large family of sequences of Markov chains \( (X_n)_{n \geq 1} \), there  is at least one nondegenerate sequence of Boolean functions which is noise stable with respect to \( (X_n)_{n \geq 1 } \), it does not tell us how to find such a sequence. Lemma~\ref{lemma: technical part of first lemma} provides us with means to do so, but it contains many parameters and does not connect to other well known definitions for graphs. We will end this paper by giving a result for reversible and irreducible continuous time Markov chains whose bottleneck ratio and spectral gap is of the same order, which says that in this case, any sequence of sets whose bottleneck ratios approximate the spectral gap well, will be noise stable. When in addition to being reversible and irreducible, the Markov chain is transitive, we show that there is a nondegenerate sequence of sets which minimize the bottleneck ratio. The most interesting thing with this result is that it provides a partial explanation of how noise stable functions can arise.

Recall that for a  random walk on a connected graph \( G \) with stationary distribution \( \pi \), the bottleneck ratio \( \Phi_* \) is defined by 
\begin{equation}
\Phi_* =\Phi_*(G) \coloneqq \inf_{A \subset V(G) \colon \pi(A)\leq 1/2} \Phi(A)
\label{equation: bottleneck graph}
\end{equation}
where \( \Phi(A) = {|E(A, A^c )|}/{\vol(A)} \). For a general continuous time Markov chain \( X \) with generator \( Q = (q_{ij}) \) and stationary distribution \( \pi\), we let \( \Phi(A) \coloneqq ({\sum_{i \in A,\, j \in A^c} \pi(i) q_{ij}})/{\pi(A)} \) and  define the bottleneck ratio as
\begin{equation}
\Phi_* = \Phi_*(X) \coloneqq \inf_{A \subset S^{(n)} \colon \pi(A)\leq 1/2} \Phi(A).
\label{equation: bottleneck Markov chain}
\end{equation}
Note that when the rate \( r = q_{ii} \) of a random walk equals one, 
\begin{equation*}
\frac{\sum_{i \in A,\, j \in A^c} \pi(i) q_{ij}}{\pi(A)} 
=
\frac{\sum_{i \in A,\, j \in A^c} \frac{\deg(i)}{\vol(G)} \cdot \frac{1}{\deg(i)}}{\sum_{i \in A} \frac{\deg(i)}{\vol(G)}} 
=
\frac{\sum_{i \in A,\, j \in A^c} 1}{\sum_{i \in A} \deg(i)}
=
\frac{|E(A,A^c)|}{\vol(A)}
\end{equation*}
i.e. in this case the two definitions coincide. 

The main reason to expect that noise stable sequences of Boolean functions should exist when \( \Phi_*^{(n)} \asymp \lambda_1^{(n)} \) is that heuristically, for a set \(A \subseteq S^{(n)}\), \( \Phi_*(A) \) estimates how large a proportion of the total mass of \( A \)  that will move to \( A^c \) during one unit of time. If \( \Phi_*^{(n)} \asymp \lambda_1^{(n)} \), then only about a proportion \( \varepsilon \) of the total mass in \( A \) will move to \( A^c \) at least once between time zero and time \( \varepsilon t_{rel}^{(n)} \). If \( A_n \) minimizes the bottleneck ratio of \( X^{(n)} \), we thus expect that the sequence \( (\mathbf{1}_{w \in A_n})_{n \geq 1} \) is noise stable, and it  only remains to show that we can find such a sequence which is in addition nondegenerate, which is why we want each Markov chain in the sequence \( (X^{(n)} )_{n\geq 1 } \) to be transitive. In Examples~\ref{example: small set bottleneck 2} and~\ref{example: small set bottleneck 3}, we show that even for random walks on connected graphs, we cannot drop the transitivity assumption. However, it is not clear whether this assumption could be replaced with something weaker.

In general, it is known that for any Markov chain, the bottleneck ratio and the spectral gap satisfies \( \Phi_*^2 \leq 2\lambda_1 \leq 4\Phi_* \), where the lower bound is attained for a unit rate random walk on \( \mathbb{Z}_n \) and the upper bound is attained for a unit rate random walk on the Hamming cube. We thus know that there are sequences of graphs with \( \Phi_*^{(n)} \asymp \lambda_1^{(n)} \), implying that the result is not void.

We are now ready to state our result.

\begin{proposition}
Let \( (X^{(n)})_{n \geq 1} \) be a sequence of transitive, reversible and irreducible continuous time Markov chains, with spectral gaps \( (\lambda_1^{(n)})_{n \geq 1 } \) and bottleneck ratios \( (\Phi_*^{(n)} )_{n \geq 1 } \).
If \( \lambda_1^{(n)} \asymp \Phi_*^{(n)} \), there exist at least one nondegenerate noise stable sequence of Boolean functions on \( S^{(n)} \).
\label{proposition: bottleneck and noise stability}
\end{proposition}

Proposition~\ref{proposition: bottleneck and noise stability} will be proven using two lemmas, which we now state and prove. In these lemmas, as well as in the rest of the notes, we will for any Boolean function  \( f\colon V(G) \to \{ 0,1 \} \) let \( A_f \) denote the set \(  \{ w \in V(G)\colon f(w) = 1 \} \). The first of these lemmas will be proven for general continuous time Markov chains, but due to issues concerning nondegeneracy which will be discussed later, the second lemma will only be proved for transitive Markov chains.

\begin{lemma}
Let \( (X^{(n)})_{n \geq 1} \) be a sequence of reversible and irreducible continuous time Markov chains with spectral gaps \( (\lambda_1^{(n)})_{n \geq 1 } \). Further,  let \( ( f_n )_{n \geq 1}\), \( f_n \colon S^{(n)} \to \{ 0,1 \} \),  be a sequence of Boolean functions with  \(    \pi (A_{f_n} ) \in (0,1/2]\). Then
\[
\lim_{\alpha \to 0 } \lim_{n \to \infty} 
P(f_n(X^{(n)}_0) \not = f_n(X^{(n)}_{t_{rel}}))
\leq
\lim_{\alpha \to 0 } \lim_{n \to \infty} 
\alpha  \, t^{(n)}_{rel}  \cdot \pi(A_{f_n}) \cdot \Phi(A_{f_n}).
\]
In particular, if \( \Phi (A_{f_n}) \asymp \lambda_1^{(n)}\), then \( ( f_n )_{n \geq 1} \) is noise stable with respect to \( (X^{(n)})_{n \geq 1} \).
\label{lemma: asymp and noise stable}
\end{lemma}

\begin{proof}
Fix \( n \geq 1 \) and let \( \alpha > 0 \) be arbitrarily chosen. By the definition of the generator \( Q_n \), for any \( h_n > 0\),
\begin{equation}
 P(X^{(n)}_{h_n} = j \mid X^{(n)}_0 =i) = h_n \, q_{ij}^{(n)} + o(h_n).
\label{eq: hn def}
\end{equation}
We will, to simplify notations, assume that \( h_n \) is choosen such that \( \alpha t^{(n)}_{rel}/h_n \) is an integer. Then
\begin{align}
P(f_n(X^{(n)}_0)=1,\,  f_n(X^{(n)}_{\alpha t^{(n)}_{rel}})=0) 
&\leq 
\sum_{k = 1}^{\alpha t^{(n)}_{rel}/h_n} P(f_n(X^{(n)}_{(k-1)h_n})=1,\, f_n( X^{(n)}_{kh_n})=0)\nonumber
\\[1ex]
 &= 
\frac{\alpha  \, t^{(n)}_{rel}}{h_n} \cdot  P(f_n(X^{(n)}_{0})=1,\,  f_n(X^{(n)}_{h_n})=0) \label{equation: only the second line}
\end{align}
where the last equality follows by stationarity. By definition, the right hand side of \eqref{equation: only the second line} equals
\begin{equation}
\alpha  \, t^{(n)}_{rel} \cdot  \sum_{i\in A_{f_n},\, j \in A_{f_n}^c} \pi_n(i) \cdot \frac{1}{h_n} \, P(X^{(n)}_{h_n} =j \mid X^{(n)}_{0} =i).
\label{equation: an equation which needs to have a name}
\end{equation}
Using~\eqref{eq: hn def}, \eqref{equation: an equation which needs to have a name} can be bounded from above by
\begin{equation*}
\begin{split}
\alpha  \, t^{(n)}_{rel} \cdot \!\!\!\! \sum_{i\in A_{f_n},\, j \in A_{f_n}^c} \!\!\!\!\!\!  \pi_n(i) \cdot \left( q_{ij}^{(n)}+\frac{o(h_n)}{h_n} \right)
&=
\alpha  \, t^{(n)}_{rel} \cdot \!\!\!\!   \sum_{i\in A_{f_n},\, j \in A_{f_n}^c} \!\!\!\!\!\! \pi_n(i) \, q_{ij}^{(n)}
+
\alpha  \, t^{(n)}_{rel} \cdot  \pi(A_{f_n})\cdot \frac{o(h_n)}{h_n} 
\\&=
\alpha  \, t^{(n)}_{rel}  \cdot \pi(A_{f_n}) \cdot \Phi(A_{f_n})
+ 
\alpha  \, t^{(n)}_{rel} \cdot  \pi(A_{f_n})\cdot \frac{o(h_n)}{h_n} .
\end{split}
\end{equation*}
In particular, we have showed that
\begin{equation}
P(f_n(X^{(n)}_0)=1,\,  f_n(X^{(n)}_{\alpha t^{(n)}_{rel}})=0) 
\leq 
\alpha  \, t^{(n)}_{rel}  \cdot \pi(A_{f_n}) \cdot \Phi(A_{f_n})
+ 
\alpha  \, t^{(n)}_{rel} \cdot  \pi(A_{f_n})\cdot \frac{o(h_n)}{h_n} .
\label{equation: upper bound by bottleneck}
\end{equation}
As \( X^{(n)} \) is reversible for each \( n \), 
\begin{equation*}
\begin{split}
P(f_n(X^{(n)}_0) \not = f_n(X^{(n)}_{\alpha t_{rel}}))
&=
P(f_n(X^{(n)}_0)=1,\,  f_n(X^{(n)}_{\alpha t^{(n)}_{rel}})=0)  + 
P(f_n(X^{(n)}_0)=0,\,  f_n(X^{(n)}_{\alpha t^{(n)}_{rel}})=1) 
\\&=
2 P(f_n(X^{(n)}_0)=1,\,  f_n(X^{(n)}_{\alpha t^{(n)}_{rel}})=0) .
\end{split}
\end{equation*}
Using the upper bound from~\eqref{equation: upper bound by bottleneck}, we thus obtain 
\begin{equation*}
\begin{split}
P(f_n(X^{(n)}_0) \not = f_n(X^{(n)}_{t_{rel}}))
&\leq
\alpha  \, t^{(n)}_{rel}  \cdot \pi(A_{f_n}) \cdot \Phi(A_{f_n})
+ 
\alpha  \, t^{(n)}_{rel} \cdot  \pi(A_{f_n})\cdot \frac{o(h_n)}{h_n} .
\end{split}
\end{equation*}
As the second term can be made arbitrarily small by choosing \( h_n \) small, the desired conclusion follows.
\end{proof}

\begin{remark}
This directly shows that if \( G_n \) is the \(n \)-dimensional Hamming cube,  the sequence \( ( f_n)_{n\geq 1} \), where \( f_n (w) = w(1) \),  is noise stable.
\end{remark}

To finish the proof of Proposition~\ref{proposition: bottleneck and noise stability}, it remains to show that given that \( X^{(n)} \) is transitive  and \( \Phi_*^{(n)} \asymp \lambda_1^{(n)} \), there is at least one \emph{nondegenerate} sequence \( (f_n)_{n \geq 1} \) of Boolean functions with \( \Phi (A_{f_n} ) \asymp \Phi_*^{(n)} \).  It is not obvious that such a sequence exists in a more general setting, and the following two examples show that even in the special case of random walks on graphs, some assumption on the graphs \( G_n \), which is stronger than \( G_n \) being regular, is needed to guarantee the existence of a  nondegenerate sequence of Boolean functions which minimizes the bottleneck ratio. In particular, these examples show that Proposition~\ref{proposition: bottleneck and noise stability} is not true for general Markov chains.

 \begin{example}
\label{example: small set bottleneck 2}
Let, as in Example~\ref{example: small set bottleneck}, \( G_n \) be the graph obtained by joining a complete graph on \( n \) vertices to a complete graph on \( n^2  \) vertices by adding a single edge. 
  
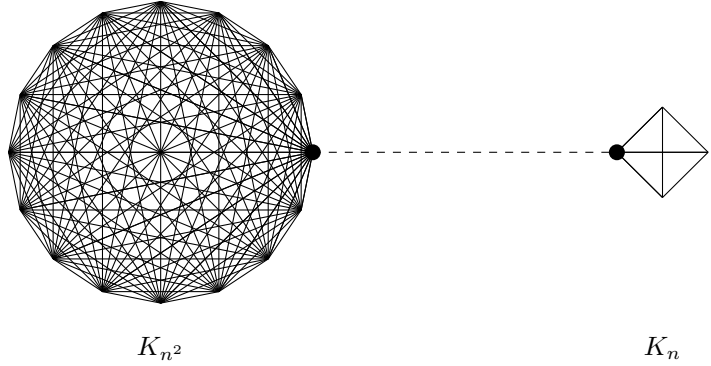
\begin{figure}[H]
 \centering
 \begin{tikzpicture}[scale = 2]

\def\N{15}
\foreach \i in {0,...,\N}{
  \foreach \j in {\i,...,\N}{
    \draw ({cos(\i*360/(\N+1))},{sin(\i*360/(\N+1))}) -- ({cos((\j+1)*360/(\N+1)},{sin((\j+1)*360/(\N+1)});
  };
};
  \fill ({cos(0*360/(\N+1))},{sin(0*360/(\N+1))})  circle (1.5pt);

\draw[dashed] (1,0) -- (3,0);

\def\M{3}
\def\scaling{0.3}
\foreach \i in {0,...,\M}{
  \foreach \j in {\i,...,\M}{
    \draw ({3+\scaling*(1-cos(\i*360/(\M+1)))},{\scaling*sin(\i*360/(\M+1))}) -- ({3+\scaling*(1-cos((\j+1)*360/(\M+1))},{\scaling*sin((\j+1)*360/(\M+1)});
  };
};

  \fill ({3+\scaling*(1-cos(0*360/(\M+1)))},{\scaling*sin(0*360/(\M+1))})  circle (1.5pt);

\draw (-0.0,-1.3) node {\(K_{n^2}\)};
\draw (3.3,-1.3) node {\(K_n\)};

 \end{tikzpicture}
 \end{figure}
 
We will once more consider the random walks on these graphs. Using the same notation as in Example~\ref{example: small set bottleneck}, we have
\[
|E(K_{n} , K_{n^2})| = 1
\]
and
\[
\vol ( K_{n} ) = n(n-1).
\]
This implies that the continuous time random walk on \( G_n \) with rate \( 1 \) has bottleneck ratio
\begin{equation*}
\Phi_* = \Phi(G_n) = \frac{|E(K_{n} , K_{n^2})| }{\vol (K_n) +1} = \frac{1}{n(n-1)+1} \asymp \frac{1}{n^2}.
\end{equation*}
However,
\[
\pi_n  \left(  K_n \right) = \frac{\vol ( K_n )+1 }{\vol(G_n)} = \frac{n(n-1)+1}{n(n-1)+n^2(n^2-1)+2} \asymp \frac{1}{n^2}
\]
i.e. this example shows that there are sequences of graphs for which the bottleneck ratio is attained by a proportion of the graphs whose measure tends to zero.
 \end{example}

Our next example shows that to assume that graphs in the sequence are regular is not enough to avoid that the bottleneck ratio is attained by an arbitrarily small part of the graph.
\begin{example}
\label{example: small set bottleneck 3}
For each \( n \geq 1 \), let \( G_n^{(1)} \) be the complete graph on \( n+1 \) vertices and \( G_n^{(2)} \) be the Hamming cube \( \mathbb{Z}_{2}^n \). Then all vertices in both \( G_n^{(1)} \) and \( G_n^{(2)} \) have degree \( n \), \( \vol (G_n^{(1)}) =(n+1)\cdot n \asymp n^2 \), \( \vol(G_n^{(2)}) = 2^{n} \cdot n \), \( \Phi_*(G_n^{(1)}) \asymp 1 \) and \( \Phi_*(G_n^{(2)}) \asymp 1/n \). Now let \( G_n \) be the graph obtained by choosing one edge in \( G_n^{(1)} \) and one edge in \( G_n^{(2)} \), removing them both and adding new edges between the ends of the chosen  edges in the respective gaphs such that the degree of each vertex is preserved and \( G_n \) is connected. Then
\[
\Phi_*^{(n)} = 2/(\vol(G_n^{(1)} )+1).
\]
Due to the order of \( \Phi_*(G_n^{(1)}) \) and \( \Phi_*(G_n^{(2)}) \) and the fact that \(\vol (G_n^{(1)})/\vol(G_n) \to 0 \), this sequence of graphs provides examples of regular graphs where the bottleneck ratio is attained by an arbitrarily small proportion of the graph.

  
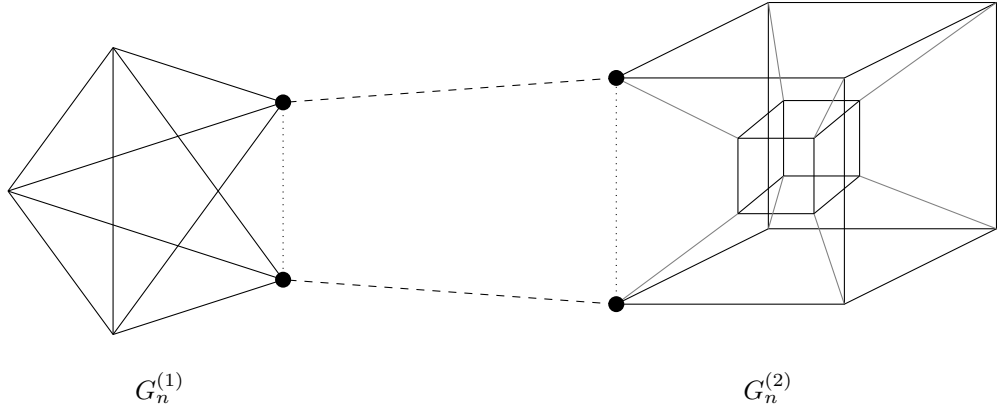
\begin{figure}[H]
 \centering
 \begin{tikzpicture}[scale = 2]

\def\N{4}
\foreach \i in {1,2,3}{
  \foreach \j in {\i,...,3}{
    \draw ({cos(\i*360/(\N+1)+180/(\N+1))},{sin(\i*360/(\N+1)+180/(\N+1))}) -- ({cos((\j+1)*360/(\N+1)+180/(\N+1)},{sin((\j+1)*360/(\N+1)+180/(\N+1)});
  };
};

\foreach \i in {0}{
  \foreach \j in {0,1,2}{
    \draw ({cos(\i*360/(\N+1)+180/(\N+1))},{sin(\i*360/(\N+1)+180/(\N+1))}) -- ({cos((\j+1)*360/(\N+1)+180/(\N+1)},{sin((\j+1)*360/(\N+1)+180/(\N+1)});
  };
};

\foreach \i in {0}{
  \foreach \j in {3}{
    \draw[dotted] ({cos(\i*360/(\N+1)+180/(\N+1))},{sin(\i*360/(\N+1)+180/(\N+1))}) -- ({cos((\j+1)*360/(\N+1)+180/(\N+1)},{sin((\j+1)*360/(\N+1)+180/(\N+1)});
  };
};

\foreach \i in {0,\N}{
  \fill ({cos(\i*360/(\N+1)+180/(\N+1))},{sin(\i*360/(\N+1)+180/(\N+1))})  circle (1.5pt);
};



\draw[dotted] (3,-0.75) -- (3,0.75);
\draw  (3,0.75) -- (4.5,0.75) -- (4.5,-0.75) -- (3,-0.75);
\draw (4,-0.25) -- (4,1.25) -- (5.5,1.25) -- (5.5,-0.25) -- (4,-0.25);
\draw (3,-0.75) -- (4,-0.25);
\draw (3,0.75) --(4,1.25);
\draw (4.5,0.75) --(5.5,1.25);
\draw (4.5,-0.75)  -- (5.5,-0.25);

\draw (3.8,-0.15) -- (3.8,0.35) -- (4.3,0.35) -- (4.3,-0.15) -- (3.8,-0.15);
\draw (4.1,0.1) --    (4.1,0.6) --   (4.6,0.6) -- (4.6,0.1) -- (4.1,0.1);
\draw (3.8,-0.15) -- (4.1,0.1);
\draw (3.8,0.35) --  (4.1,0.6);
\draw (4.3,0.35) --  (4.6,0.6);
\draw (4.3,-0.15) -- (4.6,0.1);

\draw[gray] (3.8,-0.15) -- (3,-0.75);
\draw[gray]  (3.8,0.35) -- (3,0.75);
\draw[gray]   (4.3,0.35) --(4.5,0.75);
\draw[gray]  (4.3,-0.15) -- (4.5,-0.75);
\draw[gray]  (4.1,0.1) -- (4,-0.25);
\draw[gray]  (4.1,0.6) -- (4,1.25);
\draw[gray]   (4.6,0.6) -- (5.5,1.25);
\draw[gray]  (4.6,0.1) -- (5.5,-0.25);

\fill (3,-0.75) circle  (1.5pt);
\fill (3,0.75) circle  (1.5pt);

\def\i{4}
\draw[dashed] (3,-0.75) -- ({cos(\i*360/(\N+1)+180/(\N+1))},{sin(\i*360/(\N+1)+180/(\N+1))});
\def\i{0}
\draw[dashed] (3,0.75) -- ({cos(\i*360/(\N+1)+180/(\N+1))},{sin(\i*360/(\N+1)+180/(\N+1))});

\draw (-0.0,-1.3) node {\(G_n^{(1)}\)};
\draw (4,-1.3) node {\(G_n^{(2)}\)};
 \end{tikzpicture}
 \end{figure}

\end{example}

To sum up, the examples above shows that even if we can find a sequence of sets \( A_n \) satisfying \( \Phi(A_n ) \asymp \lambda_1^{(n)}\), it is not obvious that we can find such a sequence which is in addition nondegenerate, even given quite strong conditions on the Markov chains, such as it being a random walk on a regular graph. For transitive Markov chains however, the following lemma guareantees the existence of such a sequence.

\begin{lemma}
Let \( X \) be a   transitive, reversible and irreducible continuous time Markov chain with finite state space \( S \) and stationary distribution \( \pi \). Then there is a set \( A \subset S\) with \( \pi(A) \in \left( \frac{1}{4}, \frac{1}{2} \right] \) such that \( \Phi_{*} = \Phi(A) \).
\label{lemma: nondegenericity}
\end{lemma}

\begin{proof}
Suppose that the lemma is false. Then for at least one set \(A \subset S \),
\begin{enumerate}
\item[(i)] \( \pi (A) \leq \frac{1}{4} \),
\item[(ii)] \( \Phi(A) = \Phi_* \) and
\item[(iii)] \( \Phi(A') > \Phi(A) \) for all \(A' \subseteq S \) with \( \pi(A) < \pi(A') \leq \frac{1}{2} \).
\end{enumerate}
Let \( \varphi\in \Aut(X) \) be such that \( \varphi(A) \not = A \). As \( \pi (A) < 1 \) and \( X^{(n)} \) is transitive,  at least one such function exists. As and \( \pi(A) \leq \frac{1}{4} \) by (i), we must have that  \( \pi(A \cup \phi(A)) \leq 2\pi(A) \leq \frac{1}{2} \).
By the definition of \( \Phi \), 
\begin{equation*}
  \begin{split}
    \Phi (A \cup \varphi(A)) 
    &=
    \frac{\sum_{\substack{i \in A \cup \varphi (A) \\ j \not \in A \cup \varphi(A)}} \pi(i) q_{ij}}{\pi(A \cup \varphi(A))}
    \\&=
    \frac{
      \sum_{\substack{i \in A \\ j \not \in A }} \pi(i) q_{ij}
      +
      \sum_{\substack{i \in \varphi(A) \\ j \not \in \varphi(A) }} \pi(i) q_{ij}
      -
      \sum_{\substack{i \in A \\ j  \in A^c \cap \varphi(A) }} \pi(i) q_{ij}
      -
      \sum_{\substack{i \in \varphi(A) \\ j  \in A \cap \varphi(A)^c  }} \pi(i) q_{ij}
      -
      \sum_{\substack{i \in A\cap \varphi(A) \\ j \in A^c \cap \varphi(A)^c }} \pi(i) q_{ij}
    }{
      \pi(A) + \pi(\varphi(A)) - \pi (A \cap \varphi(A))
    }
    \\&=
    \frac{
      2\sum_{\substack{i \in A \\ j \not \in A }} \pi(i) q_{ij}
      -
      \sum_{\substack{i \in A \\ j  \in A^c \cap \varphi(A) }} \pi(i) q_{ij}
      -
      \sum_{\substack{i \in \varphi(A) \\ j  \in A \cap \varphi(A)^c  }} \pi(i) q_{ij}
      -
      \sum_{\substack{i \in A\cap \varphi(A) \\ j \in A^c \cap \varphi(A)^c }} \pi(i) q_{ij}
    }{
      2\pi(A) - \pi (A \cap \varphi(A))
    }
    \\&\leq
    \frac{
      2\sum_{\substack{i \in A \\ j \not \in A }} \pi(i) q_{ij}
      -
      \sum_{\substack{i \in A\cap \varphi(A) \\ j  \in A^c \cap \varphi(A) }} \pi(i) q_{ij}
      -
      \sum_{\substack{i \in A \cap \varphi(A) \\ j  \in A \cap \varphi(A)^c  }} \pi(i) q_{ij}
      -
      \sum_{\substack{i \in A\cap \varphi(A) \\ j \in A^c \cap \varphi(A)^c }} \pi(i) q_{ij}
    }{
      2\pi(A) - \pi (A \cap \varphi(A))
    }
    \\ &=
    \frac{
      2\sum_{\substack{i \in A \\ j \not \in A }} \pi(i) q_{ij}
      -
      \sum_{\substack{i \in A\cap \varphi(A) \\ j  \not \in A \cap \varphi(A) }} \pi(i) q_{ij}
    }{
      2\pi(A) - \pi (A \cap \varphi(A))
    }
    \\ &=
    \frac{
      2 \pi(A) \Phi(A)
      -
      \pi(A \cap \varphi(A)) \Phi(A \cap \varphi(A))
    }{
      2\pi(A) - \pi (A \cap \varphi(A))
    }.
  \end{split}
\end{equation*}

As \( \pi(A) < \pi (A \cup \varphi(A)) \leq \frac{1}{2} \), by (iii) we have that
\begin{equation*}
\Phi(A) <   \Phi(A \cup \varphi(A) ) .
\label{equation: maximality assumption}
\end{equation*}
Combining this with the previous equation we obtain
\begin{equation*}
  \begin{split}
\Phi(A)
<
    \frac{
      2 \pi(A) \Phi(A)
      -
      \pi(A \cap \varphi(A)) \Phi(A \cap \varphi(A))
    }{
      2\pi(A) - \pi (A \cap \varphi(A))
    }.
  \end{split}
\end{equation*}
Rearranging, we get
\[
\Phi(A \cap \varphi(A) ) < \Phi(A).
\]
This contradicts (ii), why the conclusion of the lemma follows.

\end{proof}

\begin{proof}[Proof of Proposition~\ref{proposition: bottleneck and noise stability}]
By Lemma~\ref{lemma: nondegenericity}, there exists a nondegenerate sequence of Boolean functions \( (f_n)_{n \geq 1} \) such that \( \Phi(A_{f_n}) \asymp \lambda_1^{(n)} \). By Lemma~\ref{lemma: asymp and noise stable} this sequence is noise stable.
\end{proof}

After now having finished the proof of the main result of this section; Propostion~\ref{proposition: bottleneck and noise stability}, it is natural to ask what will happen if the assumptions of this does not hold. We have already showed what will happen if we lose the transitivity assumption, so what remains to do is to consider what happens if we drop the assumption that \( \Phi_*^{(n)} \asymp \lambda_1^{(n)} \). In the following example we will give an example of a sequence of random walks on graphs which shows that in this setting, a sequence of sets \( (A_n)_{n \geq 1} \) with \( \Phi_*(A_n) \asymp \Phi_*^{(n)} \) can be noise stable. Unfortunately, I have yet not found a Markov chain with a sequence of sets \( (A_n)_{n \geq 1 } \) which satisfies that \( \Phi_*(A_n) \asymp \Psi_*^{(n)} \not \asymp \lambda_1^{(n)} \) such that \( (\mathbf{1}_{A_n} )_{n \geq 1} \) is nondegenerate but not noise stable, neither with nor without the additional transitivity assumption.

\begin{example}
Let \( G_n  \) be the graph with vertices labeled by the integers \( 0, 1 , \ldots, 2n-1 \) and with an ende between two vertices if their label diiffer by 1 modulo \( 2n \). The random walk \( X^{(n)} \) on this graph has spectral gap \( \lambda_1^{(n)} = 1/n^2\) and bottleneck ratio \( \Phi_*^{(n)} = 1/n \) achieved by e.g. the set of vertices \( A_n \) labeled by \( \{ 0,1, \ldots, n-1 \} \). It is easy to show that the sequence \( (\mathbf{1}_{A_n})_{n \geq 1} \) is noise stable with respect to \( (X^{(n)})_{n\geq 1} \), even though \( \Phi_*^{(n)} \not \asymp \lambda_1^{(n))} \).
\end{example}

Except in special cases, the definition of the bottleneck ratio does not give us much direct information about noise stability, as, assuming that  \( A \subseteq S \) and writing \( f = \mathbf{1}_A \) 
\begin{equation*}
\begin{split}
\Phi(A) 
&= 
\frac{ \sum_{i \in A ,\, j \not \in A} \pi(i) q_{ij}}{\pi(A)}
=
\frac{1}{2} \cdot \frac{ \sum_{i,j \in S} \pi(i) q_{ij} \cdot (f(i)- f(j))^2}{\E [f]}
\\&=
 \frac{ \sum_{i \in S} \pi(i) f(i) \cdot \sum_{j \in S} q_{ij} \cdot (f(i)- f(j))}{\E [f^2]}
= 
\frac{\langle f, -Qf \rangle}{\langle f, f \rangle }
\\&=   
\frac{\sum_i \lambda_i \hat f (i)^2}{\sum_i \hat f(i)^2}
=   
\sum_i \lambda_i \cdot \frac{\hat f(i)^2}{\sum_j \hat f(j)^2} 
= 
\E_{f} [\lambda]
\end{split}
\end{equation*}
where 
\[ P_{f} (\lambda = \lambda_i) = { \frac{\hat f(i)^2}{\sum_j \hat f(j)^2}}  .
\]
If \( A_*^{(n)} \) minimizes the bottleneck ratio of \( X^{(n)} \), and we want to show that the sequence \( (f_n)_{n \geq 1} \), where \( f_n(w) = \mathbf{1}_{w \in A_*^{(n)}}(w) \), is not noise stable, we need to show that  
for all \( \delta> 0 \) there is \( k = k_\delta \) such that \( \sup_n P_{f_n} (\lambda > k \lambda_1^{(n)}) < \delta \). This is neither implied nor not implied by the fact that \( \mathbb{E}_{f_n} [\lambda] = \omega( \lambda_1) \). 
Moreover, we cannot even conclude that \( f \) is not noise sensitive, as this would require that \( \lim_{n \to \infty } P_{f_n} (\lambda<k\lambda_1^{(n)}) = 0 \) for at least one value of \( k \).
The only conclusion one can draw given that \( \Phi_*^{(n)} \not \asymp \lambda_1^{(n)} \) is that  there can be no sequence \( ( f_n )_{n \geq 1 } \) of Boolean functions in \( (\Span_{i \colon \lambda_i^{(n)} < k \lambda_i^{(n)}} \psi_i^{(n)} )_{n \geq 1 }\) for any \( k \), as we would then have
\[
\Phi_*^{(n)} =\E_{f_n} [\lambda]\leq  k\lambda_1
\]
which contradicts that \( \Phi_*^{(n)} \not \asymp \lambda_1^{(n)}  \).

\bibliographystyle{plain}
\bibliography{references}

\begin{thebibliography}{10}

\bibitem{ab2014}
Sanjeev Arora and Aditya Bhaskara.
\newblock Eigenvectors of random graphs: delocalization and nodal domains.
\newblock http://www.cs.princeton.edu/~bhaskara/files/deloc.pdf.

\bibitem{schramm2000}
Itai Benjamini, Gil Kalai, and Oded Schramm.
\newblock Noise sensitivity of boolean functions and applications to
  percolation.
\newblock {\em Publications Mathématiques de l'Institut des Hautes Études
  Scientifiques}, 90:5--43, 1999.

\bibitem{bgs2013}
Erik Broman, Christophe Garban, and Jeffrey Steif.
\newblock Exclusion sensitivity of {B}oolean functions.
\newblock {\em Probability Theory and Related Fields}, 155:621--663, 2013.

\bibitem{dll2011}
Yael Dekel, James~R. Lee, and Nathan Linial.
\newblock Eigenvectors of random graphs: Nodal domains.
\newblock {\em Random Structures Algorithms}, 39(1):39--58, 2011.

\bibitem{ds1987}
Persi Diaconis and Mehrdad Shahshahani.
\newblock Time to reach stationarity in the bernoulli-laplace diffusion model.
\newblock {\em SIAM Journal Math. Anal.}, 6(3), 1987.

\bibitem{dp2012}
Ioana Dumitriu and Soumik Pal.
\newblock Sparse regular random graphs: spectral density and eigenvectors.
\newblock {\em The Annals of Probability}, 40(5):2197--2235, 2012.

\bibitem{eff2012}
David Ellis, Yuval Filmus, and Ehud Friedgut.
\newblock A quasi-stability result for low-degree boolean functions on $s_n$.
\newblock Preprint, October 2012.

\bibitem{eff2013b}
David Ellis, Yuval Filmus, and Ehud Friedgut.
\newblock A quasi-stability result for dictatorships in $s_n$.
\newblock arXiv:1209.5557v7, December 2013.

\bibitem{eff2013a}
David Ellis, Yuval Filmus, and Ehud Friedgut.
\newblock A stability result for balanced dictatorships in $s_n$.
\newblock arXiv:1210.3989v3, May 2013.

\bibitem{esy2009}
László Erd{\H o}s, Benjamin Schlein, and Horng-Tzer Yau.
\newblock Eigenvectors on short scales and delocalization of eigenvalues for
  wigner random matrices.
\newblock {\em The annals of probability}, 37(3), 2009.

\bibitem{f2014}
Yuval Filmus.
\newblock Friedgut-{K}alai-{N}aor theorem for slices of the boolean cube, July
  2014.

\bibitem{gm1998}
Stephen Guattery and Gary Miller.
\newblock On the quality of spectral separators.
\newblock {\em SIAM Journal on Matrix Analysis and Applications},
  19(3):701--719, 1998.

\bibitem{kms2012b}
Nathan Keller, Elchanan Mossel, and Arnab Sen.
\newblock Geometric influences {II}: Correlation inequalities and noise
  sensitivity.
\newblock arXiv:1206.1210 [math.PR].

\bibitem{kms2012a}
Nathan Keller, Elchanan Mossel, and Arnab Sen.
\newblock Geometric influences.
\newblock {\em The annals of probability}, 40(3):1135--1166, 2012.

\bibitem{st1999}
Oded Schramm and Boris Tsirelson.
\newblock Trees, not cubes: hypercontractivity, cosiness and noise stability.
\newblock {\em Electronic communications in probability}, 4:39--49, 1999.

\end{thebibliography}

\clearpage
\appendix

\section{Technical lemmas}

\begin{lemma}
For each \( n \in \mathbb{N} \), let \( F_n \) be a collection of increasing functions and suppose that  
\begin{equation}\label{equation: inflimsupsup}
\inf_{(f_n)_{n\geq 1} \colon f_n \in F_n} \limsup_{x \to 0}\sup_n f_n(x) \geq \varepsilon.
\end{equation}
Then there is arbitrarily small \( x \) and \( n(x) \in \mathbb{N}\) such that \( f_{n(x)}(x) \geq \varepsilon/2 \) for all \( f_{n(x)} \in F_{n(x)} \).
\label{lemma: technical lemma on the combination of limits and sups}

\end{lemma}

\begin{proof}
Suppose that the conclusion of the lemma is false. Then for each \( x_0 > 0 \) and each \( n \geq 1 \) there is at least one function \( f_{n,x_0} \in F_{n} \) such that \( f_{n,x_0}(x_0) < \varepsilon/2 \) As each such function is increasing, we must in fact have that \(f_{n,x_0}(x) < \varepsilon/2 \) for all \( x<x_0 \). This implies that
\begin{equation*}
\begin{split} 
\inf_{(f_n)_{n\geq 1} \colon f_n \in F_n} \limsup_{x \to 0}\sup_n f_n(x)  &\leq \limsup_{x \to 0}\sup_n f_{n,x_0}(x) 
\leq \limsup_{x \to 0}\sup_n f_{n,x_0}(x_0) 
\\&\leq \sup_n f_{n,x_0}(x_0) 
\\&< \varepsilon/2.
\end{split}
\end{equation*}
As this contradicts~\eqref{equation: inflimsupsup}, the desired conclusion follows.
\end{proof}

\begin{lemma}
Let  \( A \subseteq S\) be a set with \( 0 < P(A) < 1 \). Then for any function  \( \psi \colon S \to \mathbb{R} \) with \( \E[\psi^2(w)]=1 \) and \( \E[\psi(w)]=0\),
\[
\E[\psi^{(n)}(w) \mid w \in A]^2 \leq \frac{P(w \not \in A)}{P(w \in A)^2} .
\]
\label{lemma: delocalization around zero}
\end{lemma}

\begin{proof}
Note first that
\begin{equation*}
0 = \E[\psi(w)] 
= \E[\psi(w) \mid w \in A] \cdot  P(w \in A ) 
+ \E[\psi(w) \mid w  \not \in A ] \cdot  P(w \not \in A ) .
\end{equation*}
Rewriting this, we obtain
\[
 \E[\psi^{(n)}(w) \mid w \not \in A] = - \frac{P(w \in A) }{P(w \not \in A) } \cdot \E[\psi(w) \mid w \in A].
\]
This in turn implies that
\begin{equation*}
\begin{split}
1&=\E[\psi(w)^2]
\geq
P(w \not \in A) \cdot  \E[\psi(w)^2 \mid w \not \in A] 
\geq
P(w \not \in A) \cdot  \E[\psi(w) \mid w \not \in A]^2
\\&=
P(w \not \in A) \cdot \left(\frac{P(w \in A) }{P(w \not \in A ) } \cdot \E[\psi(w) \mid w \in A] \right)^2
=
\frac{P(w \in A)^2 }{P(w \not \in A ) } \cdot \E[\psi^{(n)}(w) \mid w \in A]^2.
\end{split}
\end{equation*}
Rearranging, we obtain the desired equation.
\end{proof}

\begin{lemma}
For any positive decreasing function \( f\colon \mathbb{R}_+ \to \mathbb{R}_+ \) and any constants \( A \geq a  > 0 \),
\[
\frac{\int_0^{a/2} f(\theta) d\theta}{\int_0^A f(\theta) d\theta} \leq \frac{ f(0)}{f(0)+ f(a)}.
\] 
\label{lemma: upper bound on measure}
\end{lemma}

\begin{proof}
\begin{equation*}
\begin{split}
\frac{\int_0^{a/2} f(\theta) d\theta}{\int_0^A f(\theta) d\theta} 
&= 
\frac{\int_0^{a/2} f(\theta) d\theta}{\int_0^{a/2} f(\theta) d\theta + \int_{a/2}^a f(\theta) d\theta + \int_a^A f(\theta) d\theta}
\\&\leq
\frac{\int_0^{a/2} f(\theta) d\theta}{\int_0^{a/2} f(\theta) d\theta + \frac{a}{2}\cdot f(a) + 0}
\\&\leq
\frac{\frac{a}{2}\cdot f(0)}{\frac{a}{2}\cdot f(0)+ \frac{a}{2}\cdot f(a) + 0}
\\&=
\frac{ f(0)}{f(0)+ f(a)}.
\end{split}
\end{equation*}
\end{proof}

\begin{lemma}
Let \( \Psi \) be an eigenspace of a transitive,  reversible and irreducible continuous time Markov chain \( X \). Then for any \( w \in S \),
\[
\sup_{\psi \in \Psi \colon \langle \psi, \psi \rangle = 1} \psi(w)^2 = \dim \Psi.
\]
\label{lemma: maximal value}
\end{lemma}

\begin{proof}
Let \( \chi_1\), \ldots, \( \chi_{\dim \Psi} \) be an orthonormal basis for \( \Psi \). Then any \( \psi \in \Psi \) of length one can be written as 
\[
\psi[\theta_1, \ldots, \theta_{\dim \Psi -1}] = \Bigl( \ldots \bigl( (\chi_1 \sin \theta_1 + \chi_2 \cos \theta_1) \cdot \sin \theta_2 + \chi_3 \cos \theta_2 \bigr) \cdot \sin \theta_3 + \ldots \Bigr)
\]
where \( (\theta_1, \ldots, \theta_{\dim \Psi}) \) are the polar coordinates of \( \psi \).

We now claim that, from this representation of \( \psi \in \Psi \), it follows that for any \( w \in S \), \( \psi(w)^2 \) is maximized by 
\[
M_w \coloneqq \chi_1(w)^2 + \ldots + \chi_{\dim \Psi}(w)^2.
\]
To see this, note first that for \( \dim \Psi = 2 \), we have that
\[
\psi_1[\theta_1]  \coloneqq \psi[\theta_1] = \chi_1  \sin \theta_1 + \chi_2 \cos \theta_2 = \sqrt{\chi_1^2 + \chi_2^2} \cos (\theta_1 - \arctan \chi_2/\chi_1).
\]
This is clearly maximized by \( \sqrt{\chi_1^2 + \chi_2^2 } \) when \( \theta_1 = \arctan \chi_2/\chi_1 \). This provides the first step for an argument by induction. To see that this holds in general, define recursively
\[
\psi_m[\theta_1, \ldots, \theta_m] = \psi_{m-1}[\theta_1, \ldots, \theta_{m'-1}] \sin \theta_m + \chi_{m+1} \cos \theta_m
\]
and assume that the claim holds for\( m \in \{ 1,2, \ldots, m'-1 \} \). Then
\[
\psi_m[\theta_1, \ldots, \theta_{m'}] = \sqrt{\psi_{m'-1}^2[\theta_1, \ldots, \theta_{m'-1}] + \chi_{m}^2} \cos (\theta_{m'} - \arctan \chi_m/\psi_{m'-1}[\theta_1, \ldots, \theta_{m'-1}])
\]
is maximized by \( \sqrt{\psi_{m'-1}^2[\theta_1, \ldots, \theta_{m'-1}] + \chi_{m}^2} \) when \( \theta_{m'} = \arctan \chi_m/\psi_{m'-1}[\theta_1, \ldots, \theta_{m'-1}] \) for any choice of \( \theta_1 \), \ldots, \( \theta_{m'-1} \). This finishes the proof of the claim.

Now note that for any \( \psi \in \Psi \)   with corresponding eigenvalue \( \lambda \), and any \( \varphi \in \Aut(X) \),
\[
Q\circ \psi(\varphi(i)) = \sum_{j \in S} q_{ij} \psi(\varphi(j)) = \sum_{j \in S} q_{\varphi(i)\varphi(j)} \psi(\varphi(j)) = \sum_{j \in S} q_{\varphi(i)j} \psi(j) = \lambda \psi(\varphi(i)) 
\]
i.e. \( \psi \cdot \varphi  \Psi\). From this it follows that the maximum \( M \) must be the same for each \( w \in S^{(n)} \). This implies that
\begin{equation*}
\begin{split}
|S| \cdot M &=  \sum_{w \in S} \chi_1(w)^2 + \ldots + \chi_{\dim \Psi}(w)^2 
\\&= \sum_{w \in S}\sum_{i = 1}^{\dim \Psi} \chi_i(w)^2 
= \sum_{i = 1}^{\dim \Psi} \sum_{w \in S} \chi_i(w)^2
\\&= |S|\sum_{i = 1}^{\dim \Psi} \sum_{w \in S} \pi(w) \chi_i(w)^2
= |S|\sum_{i = 1}^{\dim \Psi} \langle \chi_i, \chi_i \rangle
\\&= |S|\sum_{i = 1}^{\dim \Psi} 1
= |S| \cdot \dim \Psi.
\end{split}
\end{equation*}
\end{proof}

\begin{lemma}
For any fixed \( u \), the functions \( \{ \psi_{i,u}\}_{i\colon  \psi_i \in \Psi} \), where  \( \psi_{i,u} (w) = \psi_i(w_u) \), are an orthonormal basis for the eigenspace spanned by \( \{ \psi_i \}_{i\colon  \psi_i \in \Psi} \)
\end{lemma}

\begin{proof}
As 
\begin{equation*}
\langle \psi_{i,u}, \psi_{j,u}\rangle = \langle \psi_{i}, \psi_{j}\rangle
\end{equation*}
it is immediately clear that \( \{ \psi_{i,u} \}_{i\colon  \psi_i \in \Psi} \) is an orthonormal set, so it remains to show that \( \psi_{j,u} \in \Span \{ \psi_i \}_{i\colon  \psi_i \in \Psi} \) for any \( j \) with \(\psi_j \in \Psi \). To obtain this result, it is enough to show that \( \psi_{j,u} \) is an eigenvector of \( -Q \) with eigenvalue \( \lambda_j \). This follows as
\begin{equation*}
\begin{split}
L \psi_{j,u} 
&= 
\sum_{k=0}^{|X_n|-1} \langle -Q \psi_{j,u}, \psi_{k,u} \rangle \psi_{k,u}
=
\sum_{k=0}^{|X_n|-1}  \langle -Q \psi_{j}, \psi_{k} \rangle \psi_{k,u}
\\&=
\sum_{k=0}^{|X_n|-1}  \langle \lambda_j \psi_{j}, \psi_{k} \rangle \psi_{k,u}
=
\sum_{k=0}^{|X_n|-1}  \lambda_j \langle \psi_{j}, \psi_{k} \rangle \psi_{k,u}
=
\lambda_j \psi_{j,u}.
\end{split}
\end{equation*}
\end{proof}

\end{document}